\RequirePackage{fix-cm}

\documentclass[smallcondensed,stropt,envcountsect,envcountsame]{article}     

\usepackage{graphicx}
\usepackage[utf8]{inputenc}
\usepackage[T1]{fontenc}

\usepackage{amsmath}
\usepackage{theorem}

\usepackage{tikz-cd}
\usepackage{mathrsfs} 
\usepackage[all,cmtip]{xy}
\usepackage{ dsfont }
\usepackage{mathtools}
\usepackage{hyperref}

\newtheorem{teo}{Theorem}[section]
\newtheorem{lem}[teo]{Lemma} 
\newtheorem{prop}[teo]{Proposition}
\theorembodyfont{\normalfont\upshape} 

\makeatletter
\DeclareRobustCommand{\qed}{%
  \ifmmode 
  \else \leavevmode\unskip\penalty9999 \hbox{}\nobreak\hfill
  \fi
  \quad\hbox{\qedsymbol}}
\newcommand{\openbox}{\leavevmode
  \hbox to.77778em{%
  \hfil\vrule
  \vbox to.675em{\hrule width.6em\vfil\hrule}%
  \vrule\hfil}}
\newcommand{\qedsymbol}{\openbox}
\newenvironment{proof}[1][\proofname]{\par
  \normalfont
  \topsep6\p@\@plus6\p@ \trivlist
  \item[\hskip\labelsep\itshape
    #1.]\ignorespaces
}{%
  \qed\endtrivlist
}
\newcommand{\proofname}{Proof}
\makeatother

\begin{document}

\title{On sheaf cohomology and natural expansions \thanks{Ana Luiza Tenorio: Supported by CAPES Grant Number 88882.377949/2019-01}
}

\author{
		{Ana Luiza Tenório, IME-USP, \small{ana.tenorio@usp.br}}
		\and
		{Hugo Luiz Mariano, IME-USP, {\small hugomar@ime.usp.br}}
	}

\maketitle

\begin{abstract}
In this survey paper, we present  \v{C}ech and sheaf cohomologies -- themes that were presented by Koszul in  University of S\~ao Paulo (\cite{koszul1957faisceaux}) during his visit in the late 1950s -- we present expansions for categories of generalized sheaves (i.e, Grothendieck toposes), with examples of applications in other cohomology theories and other areas of mathematics, besides providing motivations and historical notes. We conclude explaining the difficulties in establishing a cohomology theory for elementary toposes, presenting alternative approaches by considering constructions over quantales, that provide structures similar to sheaves, and indicating researches related to logic: constructive (intuitionistic and linear) logic for toposes, sheaves over quantales, and homological algebra.     

\end{abstract}

\section{Introduction}
\label{intro}

 Sheaf Theory explicitly began with the work of J. Leray in 1945 \cite{leray1945forme}. The nomenclature ``sheaf'' over a space $X$,  in terms of closed subsets of a topological space $X$,  first appeared in 1946, also in one of Leray's works, according to \cite{gray1979fragments}. He was interested in solving partial differential equations and build up a strong tool to pass local properties to global ones. Currently, the definition of a sheaf over $X$ is given by a ``coherent family'' of structures indexed on the lattice of open subsets of $X$ or as étale maps (= local homeomorphisms) into $X$. Both formulations emerged in the late  1940s and early 1950s in Cartan’s seminars and, in modern terms, they are intimately related by an equivalence of categories.
 
 H. Cartan proposed a concept of “coherent family” for ideals \cite{cartan1944ideaux} before Leray’s study on sheaves. His idea is more related to the development of sheaf theory in Complex Analysis, where certain conditions that hold for a point remains valid for the neighborhood of the point - as convergence properties of power series. On the other hand, the presentation of sheaves as étales spaces - due to Lazard -  is closer to Algebraic Topology: sections of étale maps compose the construction of a local section functor (explained in Section \ref{sec:5}). The global section functor (introduced in Section \ref{sec:6}) gives rise to cohomology groups with coefficients in a sheaf, which computes the obstruction from local input to global input.

We will define sheaves, using open sets, as a special kind of functor. The language of category theory will help us to deal with sheaf cohomology and allows its generalizations. However, the relation between sheaves and étales maps is important to obtain geometric intuition about the object's capacity to pass local problems to global ones through cohomology with coefficients in a sheaf.

In the 1950s, sheaves on topological spaces and their cohomology were studied by the greatest mathematicians of the time. In addition to those previously mentioned, we list J.P. Serre, A. Grothendieck, O. Zariski, and R. Godement; the latter managed to establish a standard nomenclature in his book ``Topologie alg{\'e}brique et th{\'e}orie des faisceaux'' \cite{godement1958topologie}, one of the most important references about sheaf theory so far.

At the same time, some of them were at the institute known today as Instituto de Matemática e Estatística (IME-USP), at Universidade de São Paulo. The influence of the French school on the formation of Brazilian mathematicians initiated with the arrival of A. Weil (of Weil’s conjecture and founder of Bourbaki group) in 1945 and reached O. Zariski, J. Dieudonné, J-L Koszul, A. Grothendieck, among others.  We highlight that in 1956 J-L. Koszul lectured a course about sheaves and cohomology at IME-USP, whose class notes were published in 1957 \cite{koszul1957faisceaux}, and the  A. Grothendieck’s course about topological vector spaces of 1953 was published five years later \cite{grothendieck1958espaces}.

Turning to the matter of theory’s development, J-P. Serre was the one who first applied sheaf theory in Algebraic Geometry in \cite{10.2307/1969915}, and later A. Grothendieck successfully replicated sheaf methods to spaces where the correspondent topology is not adequate. His notion of {\em topos} used a collection of morphisms in a small category  satisfying certain rules - a Grothendieck topology - to extend the notion of open covers in the definition of sheaves on topological spaces. This construction was essential to prove Weil’s conjectures, under étale cohomology (but it provided others, such as crystalline and flat cohomology), and to reformulate Algebraic Geometry. The Séminaire de Géométrie Algébrique du Bois Marie compose the enormous project that found proof for Weil’s conjectures. The project started with Bourbaki Seminars about foundations of Algebraic Geometry, published in 1962 \cite{grothendieck1962fondements} and by 1974 it was finished with  Deligne's first proof of the third conjecture \cite{deligne1974conjecture}.

What Grothendieck named \textit{topos} in \cite{grothendieck1972topos}, is nowadays known  as a \textit{Grothendieck topos} (a category of sheaves over a {\em site}, i.e.  a pair $(\mathcal{C},J)$, with $\mathcal{C}$ a small category and $J$ a Grothendieck topology); the general notion of  topos, or elementary topos, is due to the work of W. Lawvere and M. Tierney in the early 1970s. They realized that a Grothendieck topos have categorical properties that make it close to the category $Set$  of all sets and functions. For example, sheaves admit exponential objects that are analogs of the set $A^B$ of all function from $B$ to $A$, and there is an object of truth-values (subobject classifier) that, in the category $Set$, is the set $\{true,false\}$. Thus, by only assuming that a category has   a subobject classifier and satisfies some conditions (as cartesian closed to guarantee the existence of an exponential object) they reached the definition of (elementary) topos, such that any Grothendieck topos is a topos but the converse  does not hold.

Soon the study of topos theory developed many fronts. For instance, the description of an internal language (Mitchell–Bénabou language) and its Kripke-Joyal semantic, variations of Cohen’s forcing techniques using toposes, and the establishment of higher-order logic in terms of categories.

In this survey, we present sheaf cohomology and some of its possible extensions. Section 2 is devoted to preliminaries: we remind that homological algebra deals, mainly, with abelian categories (we will define this concept, but for now the reader can replace abelian categories by the category $Ab$, of all abelian groups, or $R$-$Mod$, of all modules over a given unitary ring $R$). Since most of the literature deals with specific abelian categories such as modules over rings, we provide preliminaries about cohomology for any abelian category. Besides that, we explain how to extract abelian categories from a not necessarily abelian category $\mathcal{C}$, requesting that $\mathcal{C}$ has finite limits and satisfies some other regularity properties, which is the case for the category of sheaves over topological spaces and, more generally, a Grothendieck topos. In fact, we are particularly interested in categories $\mathcal{C}$ that  we could call a ``Set-like” category, i.e, a category that keeps the basics properties and allows to perform constructions that we usually made in the category $Set$ - of all sets and functions - that play the same role of $Set$ but could be more general than just $Set$. The essential part is: since we use sets to define another structure (topological spaces, groups, rings, manifolds), we can use a $Set$-$like$ category $\mathcal{C}$ to construct other categories. In the particular case we will see, categories with abelian group structure.

In Section 3, we introduce the basics of sheaf theory, sheaf cohomology, and \v{C}ech cohomology, following the work of H. Cartan, J-L. Koszul and R. Godement.

In Section 4, we define Grothendieck toposes, exploring elementary toposes to furnish the internal logical tool, and apply it to simplify arguments in\\ Grothendieck topos cohomology.

We do not show new results, but we do point out the main ideas behind proofs that are already known and choose a presentation that allows awareness of how Grothendieck topos cohomology extends sheaf cohomology. Some demonstrations are omitted because they require excessively technical machinery (such as spectral sequences) and so would be out of our purpose of making this text a gentle introduction to topos cohomology.

In Section 5, we clarify that the current topos cohomology has issues - the definition of flabby sheaves does not work properly and the category of abelian groups over toposes that are not Grothendieck does not have enough injectives - and a strong dependence on classical logic that hinders the ``internalization'' of these notions to the intrinsic intuitionistic  (constructive) character of the toposes. We describe some attempts to address these problems, including extensions of topos cohomology over ``sheaf-like'' categories that are internally governed by an even more general logic: the linear logic.

\section{Preliminaries}
\label{sec:1}

    Summarily, a cohomology theory associates a sequence of algebraic objects to a certain space. The objects can be abelian groups, the space can be any topological space, and we can associate one with the other using chain complexes. However, the reader a bit more familiar with Homological Algebra knows that, instead of abelian groups, we can work with modules over commutative rings, vector bundles over topological spaces, or even abelian sheaves. This happens because these objects can be collected and organized in their respective categories, and these ones are examples of abelian categories. To summarize, when we work with cohomology we are working with functors into abelian categories.
    
    In this section, we present the basics of abelian categories, state the main results of Homological Algebra in this general setting, and define the notion of abelian group object - which later will provide a technique to extract abelian categories from toposes.
    
    {\em We will assume that the reader is familiar with the basic notions of category theory: category, functor, natural transformation, (co)limits,  subobjects, generators and equivalence of categories.}

\subsection{Abelian Categories}
\label{sec:2}

Let $s$ and $t$ be objects in a category $\mathcal{C}$. Recall that: If for all object $a$ in $\mathcal{C}$ there is a unique morphism $s  \rightarrow a$ then $s$ is a \textit{initial object}; if there is an unique morphism $a  \rightarrow t$, then $t$ is a \textit{terminal object}. The uniqueness property satisfied by a initial (respectively, terminal) object ensures that it is unique up to (unique) isomorphism. In case an object is simultaneously initial and terminal, it is called a zero object. After these preliminaries, we will change notation: initial and terminal objects are denoted by $0$ and $1$, respectively.

In categories with some zero object, we also have a notion of null morphism, i.e, a morphism $f: A \rightarrow B$ that factors through any  zero object (since they are pairwise isomorphic). The null morphism from $A$ to $B$ is unique and is denoted by $0_{A,B}$ or just by $0$. 

Now we can define an important concept to construct cohomology in general abelian categories.

Let $f: A \rightarrow B$ be a morphims in a category $\mathcal{C}$ with zero object. A morphism $k: K \rightarrow A$ is the \textit{kernel of $f$} if $f \circ k = 0$ and, for all morphism $h : C \to A$ such that $f \circ h = 0$, there is a unique morphism $h' : C \to K$ such that $h = k \circ h'$. Or simply, if $k$ is the equalizer of $f$ and the null morphism $0$. Diagramatically, 

\begin{center}
  \begin{tikzcd}
 K \arrow[r,"k"] \arrow[rr, controls={+(0.5,0.5) and +(-1,0.8)},"0_{K,B}"] 
& A \arrow[r, yshift=0.7ex, "f"] \arrow[r, yshift=-0.7ex, swap, "0_{A,B}"]
& B \\
C \arrow[u, dashrightarrow, "h'"] \arrow[ur, swap,"h"] \arrow[urr, bend right, swap,"0_{C,B}"] 
\end{tikzcd}
  \end{center}

The cokernel of $f$ is defined dually, i.e., by the interchanging  of source and target of all arrows. Throughout the text, we will denote the morphism $k$, kernel of $f$, by $ker(f)$, and the object $K$ associated to it by $Ker(f)$. Analogously, for the cokernel.

If the set of morphisms $Hom(A,B)$ of a category $\mathcal{C}$  has the structure of an abelian group, and the composition of morphisms is bilinear, then $\mathcal{C}$ is an $Ab$-category (or \textit{preadditive}). Here $Ab$ is the category of abelian groups and we adopted the nomenclature of $Ab$-category to familiarize the reader with the idea of an enriched category. In this case, $\mathcal{C}$ is enriched over $Ab$, so $Hom(A,B)$ is an object in $Ab$, for every $A, B$ objects in $\mathcal{C}$.

Examples of $Ab$-categories are the categories $Ab$, of all abelian groups and its homomorphisms, and  $R$-$Mod$, of all left modules over a ring $R$ and homomorphisms. More intricate examples came from categories useful in homological algebra whose objects are complexes of abelian groups, complexes of modules over a ring and  filtered modules over a ring. Moreover, every triangulated category\footnote{This is an important category in the study of Homological Algebra, but much more sophisticated than the previously examples. We will not explain it in the survey.} is an $Ab$-category.

A \textit{biproduct}\label{biprod-def} is a quintuple $(P,p_A,p_B,s_A,s_B)$ such that:
\begin{center}
       $p_A: P \rightarrow A$, $p_B: P \rightarrow B$, $s_A: A \rightarrow P$ and $s_B: 		 B \rightarrow P$ satisfies the equations: $p_A \circ s_A = id_A$, $p_B \circ s_B = 		 id_B$, $p_A \circ s_B = 0$, $p_B \circ s_A = 0$ and  $s_A \circ p_A + s_B \circ 		p_B = id_P$  
\end{center}

We observe that in $Ab$-categories the existence of biproduct is equivalent to the existence of product and, also, the existence of coproduct. 

When a category is an $Ab$-category that has biproducts and a zero object, then it is called an \textit{additive} category. So, with the abelian groups structure in the set of morphisms, we obtain that the zero of $Hom(A,B)$ coincides with the null morphism $0_{A,B}$. See a demonstration of this in \cite[Chap. 1.2]{borceux1994handbook}. This is  an interesting property since additive categories have to indirectly handle null morphisms, through kernels and cokernels.

An additive category  $\mathcal{C}$ is an \textbf{abelian category} if it satisfies also the following conditions:
\begin{enumerate}
    \item[AB1] Every morphism has kernel and cokernel.
    \item[AB2] Every monomorphism is a kernel and every epimorphism is a cokernel.
\end{enumerate}

Except for filtered modules over a ring, and, in general, triangulated categories, all the given examples of $Ab$-categories are also abelian categories. Additionally, an important exemplar of an abelian category is the category of abelian sheaves (see Theorem \ref{abE-th}). 

Note that $AB1$ allow us to construct kernels of cokernels and vice versa. Furthermore, for any morphism $f$ in an abelian category, we have:
\begin{center}
    \begin{tikzcd}
Ker (f) \arrow[r, "ker(f)", tail] & A \arrow[r, "f"] \arrow[d, "coker(ker(f))"', two heads] & B \arrow[r, "coker (f)", two heads] & Coker (f) \\
 & Coker(ker(f)) \arrow[r, "\Bar{f}", dashed] & Ker(coker(f)) \arrow[u, "ker(coker(f))"', tail] & 
\end{tikzcd}
\end{center}

It is not difficult to see there is a unique $\Bar{f}: Coker(ker(f))  \to Ker(coker(f))$ which makes the above square commutative. An important observation here is that the $AB2$ axiom is equivalent to the condition of $\Bar{f}$ be an isomorphism. Defining $Im(f) = Ker(coker(f))$ and $Coim(f) = Coker(ker(f))$, we can say that asking for $\Bar{f}$ be an isomorphism is the same as asking for the validity of the Fundamental Homomorphism Theorem, in an abstract form.

Given an abelian category $\mathcal{C}$ we can add $ABn$ axioms. In this survey, the important one is $AB5$.
\begin{enumerate}
    \item[AB3] Given a family $\{A_i\}_{i \in I}$ of objects in $\mathcal{C}$, then the direct sum $\bigoplus\limits_{i \in I} A_i$ exists.
    \item[AB4] The $AB3$ axiom holds and direct sum of family of monomorphisms also is a monomorphism.
    \item[AB5] The $AB3$ axiom holds and if $\{A_i\}_{i \in I}$ is a direct family of subobjects of an object $A$ in $\mathcal{C}$, and $B$ is any subobject of $A$, then $(\sum\limits_{i \in I}A_i)\cap B = \sum\limits_{i \in I}(A_i \cap B)$, where the capital-sigma denotes sup of $A_i$, and the intersection denotes inf of subobjects.
\end{enumerate}

The $AB5$ will be central because of the following Grothendieck's Theorem:

\begin{center}
  If an abelian category satisfies $AB5$ and has a generator then it has enough injectives \cite[Theorem 1.10.1]{grothendieck1957}
\end{center} \label{grotheorem}

We will use this Theorem to show that the abelian categories defined from Grothendieck toposes are good enough to develop a cohomology theory. Now, let's return to our preliminaries.

Most part of the simplest  notions and proofs that occur in ``concrete'' abelian categories (as $Ab$, or $R$-$Mod$) can be reproduced in general abelian categories by the systematic use of universal properties. However, more sophisticated  results in Homological Algebra demand specific techniques to handle the fact that we do not know what kind of structure the objects have. For example, when we consider the category $Ab$, we know that the objects are abelian groups. However, if we have to deal with an arbitrary abelian category this information is not (directly) available. To manage this delicate scenario there are at least two ways of proving results regarding general abelian categories.

The simplest technique is apply the Freyd-Mitchell embedding Theorem -  originally stated in \cite[Theorem 7.34]{freyd1964abelian} and reformulated in modern terms in \cite[ Theorem 1.6.1]{weibel1994introduction} - that guarantees we can fully embed {\em small} abelian categories into the category $R$-Mod, for some ring $R$. Roughly speaking, this means all morphisms that exist in $R$-$Mod$, such as kernels and cokernels (quocients), and all diagram chasing that can be done in $R$-$Mod$, still holds for the correspondent small abelian category. We recommend \cite{stacks-project} for more results related to the Freyd-Mitchell embedding Theorem. 

Nevertheless, there are non small abelian categories so we may use a stronger but more complex technique: construct the notion of \textit{pseudoelement}, as nominated in \cite{borceux1994handbook} (or \textit{generalized element}, as in \cite{zbMATH01216133}). This technique enable us to do ``half of the job with elements''. More precisely, to check by this simulation of ``elements'' if some candidate arrow, previously build by combined applications of universal properties, indeed satisfies some desired property.

Once mentioned these two techniques, we state that the famous snake lemma holds in any abelian category \cite[Section 1.10]{borceux1994handbook}. We will skip even state it here, but we highlight that, as a general ambient for Homological Algebra, abelian categories were built so that the Snake Lemma is valid. When Cartan's and Eilenberg's book ``Homological Algebra'' appeared in 1956 \cite{eilenberg1956homological} the theory was exhibited for categories of modules over rings but was also know that it could be replicated for other structures, for instance, abelian sheaves. This motivated A. Grothendieck - and not only him - to define abelian categories and establish Homological Algebra for it in \cite{grothendieck1957}. Nowadays, we have even more general categories where the Snake Lemma holds, for instance, the \textit{homological categories} \cite{borceux2004mal}  (observe that the lemma hold for non-abelian groups, even so they do not form an additive category \cite[Chap 1.2]{borceux1994handbook}). However, abelian categories still are the most common scenario for the study of (co)homologies.

\subsection{Homological Algebra}
\label{sec:3}
The reader familiar with Homological Algebra techniques for a particular abelian category (as presented in \cite{weibel1994introduction}, for example) can skip this subsection. However, if there is a curiosity to see how to construct cohomology in the abstract setting of any abelian category, we exhibit the modifications that have to be done to define the basics concepts. Here we state without proof results that are needed in Sections \ref{sec:6} and \ref{sec:12}.

For any abelian category $\mathcal{C}$ we define a \textit{cochain complex} by taking sequences $\{C^q\}_{q \in \mathds{Z}}$ of objects in $\mathcal{C}$, and endow it with coboundary morphisms $d^q_C: C^q \rightarrow C^{q+1}$ such that $d^{q+1} \circ d^q = 0$, for all $q \in \mathds{Z}$. A cochain complex is denoted by $C^{\bullet},$ and we establish morphisms of complexes $h^{\bullet}: C^{\bullet} \rightarrow D^{\bullet}$ with a colection of morphism $h^q: C^q \rightarrow D^q$ such that $h^{q+1} \circ d^q_C = d^{q+1}_D \circ h^q$, for all $q \in \mathds{Z}.$ Observe that, with coordinatewise composition and identities, this forms a category $Ch(\mathcal{C})$, called category of cochain complex of $\mathcal{C}$, and it is as abelian category whenever $\mathcal{C}$ is abelian \cite[Theorem 1.2.3]{weibel1994introduction}. 

Since $d^{q} \circ d^{q-1} = 0$, we have that $0 \subseteq Im(d^{q-1}) \subseteq Ker(d^{q}) \subseteq C^q$. So it is possible to define the $q$-th cohomology object of $C^\bullet$ by $$H^q(C^{\bullet}) = Ker(d^{q})/Im(d^{q-1}) = Coker(Im(d^{q-1}) \rightarrow Ker(d^{q})).$$
where the image is $Im(d^{q-1})=Ker(coker(d^{q-1}))$.

As the reader may suspect, a sequence of objects in an abelian category $\dots \xrightarrow{f_{q-1}} A_q \xrightarrow{f_q} A_{q+1} \xrightarrow{f_{q+1}} \dots$ is exact if $Ker(f^q) = Im(f^{q-1})$. So cohomology measures the failure of exactness in the cochain complex.

Let $f^{\bullet} : C^{\bullet} \rightarrow D^{\bullet}$ be a complex morphism. Since we are working with arbitrary abelian categories, define an induced morphism  $H^q(f) : H^q(C^\bullet) \to H^q(D^\bullet)$, $q \in \mathds{Z}$, is more complicated than usual, but lets describe the idea:

For each $q \in \mathds{Z}$, a morphism $f^q : C^q \to D^q$ restricts to $$f^q_K : Ker(d_C^{q}) \to Ker(d_D^{q}) \mbox{ and to } f^q_I : Im(d_C^{q-1}) \to Im(d_D^{q-1})$$ The above follows directly from diagram chases, by the universal properties of kernels and cokernels. The coboundary morphism also provides a morphism $\alpha^{q}:Im(d_C^{q-1}) \rightarrow Ker(d_C^{q})$ such that $Coker(\alpha^q_C) = H^q(C^{\bullet})$. By the universal property of cokernel, there is a unique morphism $Coker(\alpha^q_C) \rightarrow Ker(d^{q}_D)$ as below:

\begin{center}
    \begin{tikzcd}
Im(d^{q-1}_C) \arrow[d,"f^q_I"]  \arrow[r, "\alpha^q_C", tail] & Ker( d^{q}_C) \arrow[d,"f^q_K"] \arrow[r, two heads] & Coker  (\alpha^q_C) \arrow[ld, dashed] \\
Im(d^{q-1}_D) \arrow[r, "\alpha^q_D", tail]  & Ker( d_D^{q}) \arrow[r, two heads] &  Coker  (\alpha^q_D)
\end{tikzcd}
\end{center}

Completing the bottom part of this diagram with the cokernel of $\alpha^q_D$ we obtain a unique morphism $H^q(C^{\bullet}) \cong Coker(\alpha^q_C) \rightarrow Coker(\alpha^q_D) \cong H^q(D^{\bullet}) $. This induced morphism is  $H^q(f) : H^q(C^\bullet) \to H^q(D^\bullet)$. Clearly, the mapping $f \mapsto H^q(f)$ determines a (covariant) functor $H^q : Ch({\cal C}) \to {\cal C}$, for all $q \in \mathds{Z}$. 

Given two complex morphisms $f^{\bullet},g^{\bullet}: C^{\bullet} \rightarrow D^{\bullet},$ they are called \textit{homotopic} if, for each $q \in \mathds{Z}$, there is $h^q: C^q \rightarrow D^{q-1}$ (called \textit{cochain homotopy}) such that $f^q - g^q = d^{q-1}_D\circ h^q + h^{q+1}\circ d^q_C$. Chain homotopies are important because they relate two different morphisms through their induced maps in the cohomology objects. More precisely:
\begin{prop}\label{homotopic}
If $f^{\bullet}$ is homotopic to $g^{\bullet}$, then $H^q(f^{\bullet}) \cong H^q(g^{\bullet})$, for all $q \in \mathds{Z}$.
\end{prop}

Now, we introduce exact functors. First, a (covariant) functor $F: \cal{C} \to \mathcal{C}'$ between abelian categories is \textit{additive} if the map that sends morphisms $f$ in $\mathcal{C}$ to morphisms $F(f)$ in $\mathcal{C}'$ is a homomorphism of groups. 

Then, given an exact sequence $0\to A \to B \to C \to 0$ in $\mathcal{C}$, we say $F$ is 
\begin{enumerate}
    \item \textit{exact} if $0\to F(A) \to F(B) \to F(C) \to 0$ is an exact sequence;
    \item \textit{left exact} if $0\to F(A) \to F(B) \to F(C)$ is an exact sequence;
    \item \textit{right exact} if $ F(A) \to F(B) \to F(C) \to 0$ is an exact sequence.
\end{enumerate}

Two important examples of left exact functors are $Hom_{\mathcal{C}}(-,A)$ (contravariant case) and $Hom_{\mathcal{C}}(A,-)$ (covariant case), for $A$ a fixed object of $\mathcal{C}$.

Now, remember that an object $I$ in an abelian category is \textit{injective} if for all morphism $\alpha: A \rightarrow I$ and all monomorphism $m : A \rightarrow B$, there is at least one morphism $\beta: B \rightarrow I$ such that $\alpha = \beta \circ m$  (equivalently, $I$ is injective if and only if the functor $Hom(-,I)$ is exact). A  \textit{resolution $A \rightarrow I^{\bullet}$ of an object} $A$ is an exact sequence $0 \rightarrow A \rightarrow I^0 \rightarrow I^1 \rightarrow ... $; this resolution is an \textit{injective resolution} if $I^i$ in injective for each $i \geq 0$. If an abelian category \textit{has enough injectives}, then any of its objects $A$ admits some injective resolution. Dually, we define projective objects and projective resolutions.

The concept of enough injectives is central in homological algebra because of the following theorem.

\begin{teo}\label{theo:derivedfunctor}
Let $\mathcal{C}$ and $\mathcal{C'}$ abelian categories, with $\mathcal{C}$ having enough injectives, and let $F: \mathcal{C} \rightarrow \mathcal{C'}$ be a (covariant) left exact additive functor. Then:
\begin{enumerate}
    \item[(i)] There are additive  functors $R^qF : \mathcal{C} \rightarrow \mathcal{C'} $ for all $q \geq 0$;
	    \item[(ii)] There is an isomorphism $F \cong R^0F$;
	    \item[(iii)] For each exact sequence $E: 0 \rightarrow A_1  \rightarrow A_2  \rightarrow A_3  \rightarrow 0 $ and each $q \geq 0$, there is a morphism $\delta^q_E: R^qFA_3 \rightarrow R^{q+1}FA_1$ that makes the following sequence exact
	    $$ \dots \rightarrow R^{q}FA_1 \rightarrow R^{q}FA_2 \rightarrow R^{q}FA_3 \xrightarrow{\delta^q_E} R^{q+1}FA_1 \rightarrow \dots$$
	    \item[(iv)] The morphisms $\delta^q_E$ are natural in $E$.
\end{enumerate}
\end{teo}

These $R^qF : \mathcal{C} \rightarrow \mathcal{C'} $ functors are unique up to natural isomorphisms, they  are called \textit{$q$-th right derived functor of $F$} and $R^qF(A) \cong H^qF(I^{\bullet})$, where $I^{\bullet}$ is a resolution of $A$. 

It is worth to mention that the famous $Ext(-,A)$ functor is the derived functor of $Hom_{\mathcal{C}}(-,A)$. Since $Hom_{\mathcal{C}}(-,A)$ is exact iff $A$ is injective, and $Ext$ measures how far $Hom_{\mathcal{C}}(-,A)$ is from be an exact functor, we can say $Ext$ measures the failure of $A$ in being injective.

We introduce a last definition in this section: let $F: \mathcal{C} \rightarrow \mathcal{C'}$ as in the above theorem. An object $A$ of $\mathcal{C}$ is \textit{$F$-acyclic} (or acyclic for $F$) if $R^qF(A) = 0$ for all $q > 0$.

\textbf{Remark:} This definition also describes a way to measure the failure of a sequence to be exact, so we could define derived functors using acyclic objects instead of injective ones.

So far, we discussed that if $\mathcal{C}$ is an abelian category, then we can define \textit{cohomology objects} of its correspondent cochain complex $\mathcal{C}^{\bullet}$, and  several constructions and results of Homological Algebra are available. However, what is a \textit{cohomology theory}? That is, for different abelian categories (and even not abelian categories) what guarantees we are dealing with a cohomological structure? The answer is: the Eilenberg–Steenrod axioms. 

The Eilenberg–Steenrod axioms state that a collection of functors form a (co)homology theory if it satisfies a certain list of axioms, for fixed coefficients (we will see cohomologies where the coefficients are sheaves, but the reader can think, for instance, in the singular cohomology with coefficients in a fixed abelian group). Moreover, we may obtain other types of cohomology theories if we remove one of the axioms; in particular, the removal of the dimension axiom provides a ``generalized (co)homology theory'', which is the case, of some $K$-theories. In other words, ``cohomology'' has a broad application. It is interesting observe that different cohomologies may coincide for suitable choices of spaces and coefficients (see Theorem \ref{Cech-te}, and Sections \ref{sec:8} and \ref{sec:14}).

\subsection{Abelian Group Object}
\label{sec:4}
	If $\mathcal{C}$ is a category with binary products and terminal object $1,$ we can define the notion of \textbf{group object in $\mathcal{C}$} as an object $G$ in $\mathcal{C}$ equipped with morphisms 
\begin{center}
\begin{tikzcd}
e: 1 \arrow[r] & G & i: G \arrow[r] & G & m:G \times G \arrow[r] & G
\end{tikzcd}
\end{center}
in $\mathcal{C}$, such the following diagrams commute
\begin{center}
\begin{tikzcd}
G \times G \times G \arrow[r, "id_G\times m"] \arrow[d, "m \times id_G"'] & G \times G \arrow[d, "m"] & 1\times G \arrow[r, "e \times id_G"] \arrow[rd, "\cong"'] & G \times G \arrow[d, "m"] & G\times 1 \arrow[l, "id_G \times e"'] \arrow[ld, "\cong"] \\
G \times G \arrow[r, "m"'] & G &  & G & 
\end{tikzcd}

\begin{tikzcd}
G \arrow[d, "!"] \arrow[r, "\triangle"] & G \times G \arrow[r, "i\times id_G"] & G \times G \arrow[d, "m"] \\
1 \arrow[rr, "e"'] &  & G
\end{tikzcd}

\begin{tikzcd}
G \arrow[d, "!"] \arrow[r, "\triangle"] & G \times G \arrow[r, "id_G\times i"] & G \times G \arrow[d, "m"] \\
1 \arrow[rr, "e"'] &  & G
\end{tikzcd}

\end{center} 

The morphism $\triangle = (id_G,id_G): G \rightarrow G \times G$ is the diagonal morphism. Note that these diagrams are expressing the group axioms. If we want to add an abelian condition and form an abelian group object, then we must include 
\begin{center}
\begin{tikzcd}
G \times G \arrow[r, "\tau"] \arrow[rd, "m"'] & G \times G \arrow[d, "m"] \\
 & G
\end{tikzcd}
\end{center}
commutative, where $\tau = (\pi_2, \pi_1): G \times G \rightarrow G \times G $ is the twist morphism.

So an \textit{abelian group object} is a quadruple $(G,e,i,m),$ where the diagrams above commute, and the \textit{category} $Ab(\mathcal{C})$ \textit{of abelian groups object} in $\mathcal{C}$ is the category defined over the base category where the objects are abelian groups objects in $\cal{C}$ and the morphisms are morphisms in $\mathcal{C}$ that commute with the corresponding morphisms $e, i$, and $m.$ In more details, if $\mathcal{G} =(G,e,i,m)$ and $\mathcal{G}'=(G',e',i', m')$ are abelian group objects in $\mathcal{C}$, then an arrow $h : G \to G'$ in $\mathcal{C}$ determines an arrow $h : \cal G \to \cal G'$ in $Ab(\mathcal{C})$ iff $h\circ e = e'$, $h \circ i = i' \circ h$ and $h \circ m = m' \circ (h \times h)$.

Since this internal notion of group uses only products and commutative diagrams in the category $\mathcal{C}$, it follows easily that the forgetful functor $E: Ab(\mathcal{C}) \to \mathcal{C}$ creates limits.

Two notable examples of group objects are topological groups, when $\mathcal{C}$ is the category of topological spaces, and Lie groups, when $\mathcal{C}$ is the category of smooth manifolds. The base category $\mathcal{C}$ will be a topos throughout this survey. 

\section{Sheaves}
\label{sec:5}

Interested in fixed points results applied to the realm of partial differential equations, Jean Leray published in 1945, while a prisoner in the 2nd World War, the paper \cite{leray1945forme} that would originate sheaf theory. He published a more refined paper about sheaf theory and spectral sequences in 1950 \cite{leray1950anneau}, with the original ideas preserved. Meanwhile, Henri Cartan starts the Séminaire at the École Normale Supérieure, and reformulates sheaf theory. Also in 1950, in the third year of this seminar, sheaves appear as what is now know as ``étalé spaces''. Results using sheaf methods were increasingly showing up, but the terminology was not established. It was Roger Godement who achieve a standard language for the theory (for example, presheaves are functors, sheaves are a special kind of presheaves; the notion of sheaf in Cartan's seminars pass to be nominated an étalé space) with his book published in 1958 \cite{godement1958topologie}. 

Less about the history and more about the philosophy of sheaf theory: since the beginning, there was some notion that allows pass local data to global data. In the work of Godement, the flabby sheaves were responsible to play this role, while Grothendieck worked more with injective sheaves. The idea is that the cohomology groups obtained from resolutions of this specific kind of sheaves are trivial, so we do not have obstructions from local to global. The power of sheaf theory is to provide machinery to solve global problems by resolving them locally, which is especially interesting for Algebraic Geometry and Complex Analysis.

Let $X$ be a topological space. We denote by $\mathcal{O}(X)$ the category associated to the poset  of all open sets of $X$. A \textit{presheaf of sets} is a (covariant) functor $F: \mathcal{O}(X)^{op} \rightarrow Set$, and a morphism of presheaves is a natural transformation. Given inclusions $U \subseteq V$, we use $s_{|^V_U}$ (or just $s_{|_U}$) to denote the ``restriction map'' from $F(V)$ to $F(U)$. 

If $U \subseteq X$ is open and $U = \bigcup\limits_{i\in I} U_i$ is an open cover, a presheaf $F$ is a \textit{sheaf} (of sets) when we have the following diagram
\begin{center}
     \begin{tikzcd}
F (U) \arrow[r, "e"] & \prod\limits_{i\in I}F (U_i) \arrow[r, "p", shift left=1 ex] 
\arrow[r, "q"', shift right=0.5 ex]  & {\prod\limits_{(i,j) \in I \times I}F (U_i \cap U_j)}
\end{tikzcd}
 \end{center}
 
 is an equalizer in the category $Set$, where:
 \begin{enumerate}
     \item $e(t) = \{t_{|_{U_i}} \enspace | \enspace i \in I\}, \enspace t \in F (U)$ 
     \item     $p((t_k)_{ k \in I}) = (t_{i_{|_{U_i \cap U_j}}})_{(i,j)\in I\times I}$ \\ $q((t_k)_{k \in I}) = (t_{j_{|_{U_i \cap U_j}}})_{(i,j)\in I\times I}, \enspace (t_k)_{k \in I} \in \prod\limits_{k\in I}F (U_k)$
 \end{enumerate}\label{sheaf}

This definition is useful to understand categorical properties and provide a simple way to visualize its generalization when we substitute $\mathcal{O}(X)$ by an arbitrary category. However, there is an equivalent and more concrete form to describe a sheaf. Instead of presenting an equalizer diagram, we say that the preasheaf $F$ satisfies two conditions: 
\begin{enumerate}
    \item \textbf{(Gluing)} If $s_i \in F(U_i)$  is a \textit{compatible family}, i.e., $s_{i_{|_{U_i \cap U_j}}} = s_{j_{|_{U_i \cap U_j}}}$ for all $i,j \in I$, there is some  $s \in F(U)$ such that $s_{|_{U_i}} = s_i, i \in I$. We say $s$ is the \textit{gluing} of the compatible family.
    \item \textbf{(Separability)} Given $s, s' \in F(U)$ such that $s_{|_{U_i}} = s'_{|_{U_i}},$ for all $i \in I$, $s = s'$.
\end{enumerate}

A morphism of sheaves is a morphism of presheaves, that is, a natural transformation between functors, and it is clear that this defines a category, denoted by $Sh(X)$. Note that in the definition of sheaves we could replace $Set$ by any category with all small products, for example, the category of abelian groups $Ab$, and in this case we change the nomenclature to \textit{abelian sheaves}\label{absheaves}. We will return to this in Section \ref{sec:6}.

If $F$ is a presheaf, the \textit{stalk of $F$ at the point $x \in X$} is the direct limit $F_x := \varinjlim\limits_{U \in \mathcal{U}_x} F(U)$, where $\mathcal{U}_x = \{U \in {\cal O}(X): x \in U\}$ is the poset of open neighborhoods of $x$. A presheaf $F$ satisfies the separability condition above if and only if the canonical morphisms $F(U) \to \prod\limits_{x \in U}  F_x$, $U \in {\cal{O}}(X)$ are monomorphisms. We will see in the next paragraphs that stalks are important to transform presheaves into sheaves.

Now we can say that sheaves capture global information from the gluing of local properties. For example, given an open subspace $U$ of a topological space $X$, and an open cover $U = \bigcup\limits_{i\in I} U_i$, there is a functor,  ${\cal C}_\mathds{R}$, that takes opens $U$ in $X$ and sends to the set ${\cal C}_\mathds{R}(U)= \{f: U \rightarrow \mathds{R} \,|\, f$ is a continuous function$\}$. Since the restriction of a continuous function to a subset of its domain is still a continuous function,  ${\cal C}_{\mathds{R}}$ is a presheaf. Since $f_i(x) = f_j(x), \forall x \in U_i \cap U_j,$ there is a unique function $f$ such that $f_{|_{U_i}} = f_i$. Besides that, the continuity of the $f_i$'s implies  the continuity of the gluing $f$, so $f \in {\cal C}_{\mathds{R}}(U)$. Analogously, the presheaves of differential, smooth, or analytic functions are sheaves \cite{tennison_1975}.

This example may remind the reader of germs and stalks over points in a topological space with respect to étale bundles (local homeomorphisms) and this is not only a coincidence: for any continuous function $p: E \rightarrow X$ we define $\Gamma_p(U) = \{s: U \rightarrow E \enspace| s$ is continuous and $\enspace p(s(x)) = x, \forall x \in U\}$ and is possible to prove that $\Gamma_p$ is a sheaf, called \textit{sheaf of sections of the continuous function $p$}. Moreover, if $F$ is sheaf over a topological space $X$, taking $E_F :=  \coprod\limits_{x\in X}F_x$ the disjoint union of stalks of $F$ for each point $x$ in $X$, and  defining an adequate topology in $E_F$, the (obvious) projection function  $p_F : E_F \rightarrow X$ determines a local homeomorphism: this leads to a natural isomorphism between $F$ and $\Gamma(p_F)$. So every sheaf over $X$ is (naturally isomorphic to) the sheaf of sections of a local homeomorphism over $X$. Sheaf Theory inherits the nomenclature of constructions involving étale bundles because the two notions are strongly related through the category equivalence between the category of étale bundles over $X$ and the category of sheaves over $X$, for each topological space $X$. The reader can find a detailed account on this subject in \cite[Chap. II]{maclane1992sheaves}.

The spatial-functorial identification process described above is useful to provide the ``best sheaf approximation of a given presheaf'' as follows: any presheaf $F : {\cal O}(X)^{op} \to Set$, can be ``sheafificated'' into $a(F) := \Gamma(p_F) : {\cal O}(X)^{op} \to Set$ above $F$, i.e.  $a(F)$ is a sheaf over $X$ and there is a natural transformation $\eta_F : F \to a(F)$ that is initial among the natural transformations $\sigma : F \to S$, where $S$ is a sheaf over $X$; moreover, the stalk  of $a(F)$ at a point $x \in X$ is isomorphic to the stalk $F_x$. For instance, given  a set $A$, the ``constant  presheaf'' with value $A$ is the contravariant functor   $F_A(U \hookrightarrow V) = (A \overset{id_A}\leftarrow A)$; its stalk at a point $x \in X$ is isomorphic to $A$ and its sheafification, $a(F_A) : {\cal O}(X)^{op} \to Set$, is isomorphic to the sheaf of continuous function with value $A$ (viewed as a discrete topological space): ${\cal C}_{A}(U) = \{ f : U \to A \, | \, f$ is a continuous function$\}$, $U \in {\cal O}(X)$. 

Another relevant example of sheaf came from Commutative Algebra and it is central for the development of modern Algebraic Geometry: for each commutative unitary ring $R$, there is a canonical sheaf, ${\cal O}_R$, of rings defined over its prime spectrum space\footnote{$Spec(R) = \{ p \subseteq R: p$ is a proper prime ideal of $R\}$, and it is endowed with the so called ``Zariski Topology''.}, $Spec(R)$, this sheaf is determined on a (canonical) basis of the (spectral) topology of $Spec(R)$ just taking adequate localizations of the ring $R$; the stalk of this sheaf at a proper prime ideal $p \in Spec(R)$ is isomorphic to the local ring $R_p = R[R\setminus p]^{-1} $. The pair $(Spec(R), {\cal O}_R)$ is called the affine scheme associated to $R$; we will return to this example later, in Section \ref{sec:8}.

\subsection{Sheaf Cohomology}
\label{sec:6}

In this section, we present the subject ``Sheaf Cohomology'' in the usual way, omitting proofs that can be easily found in the literature, as in \cite{grothendieck1971revetement,godement1958topologie}, but providing intuition about the associated ideas. Our aim here is to list some results of this theory that will reappear in the next section with the appropriate modifications. 

For the reader's convenience, we start explaining why we can do sheaf cohomology in $Ab(Sh(X))$, i.e., how abelian sheaves are equivalent to abelian groups objects of $Sh(X)$

 Note that abelian presheaves $\mathcal{O}(X)^{op} \rightarrow Ab$   form the category of functors $Ab^{\mathcal{O}(X)^{op}}$. Then, for every functor $F$ that is an object in $Ab^{\mathcal{O}(X)^{op}}$, we have that $F(U)$ is an abelian group for every $U \in \mathcal{O}(X)$. So, for each $U \in {\cal{O}}(X)$, there are $m_U: (F \times F)(U) \cong F(U)\times F(U) \rightarrow F(U)$, $i_U:F(U) \rightarrow F(U)$, and $e_U: 1 \rightarrow F(U)$ such that they determine natural transformations and  the diagrammatic rules of abelian group object holds, i.e., $F$ is an abelian group object of $Set^{\mathcal{O}(X)^{op}}$.  On the other hand, if $G \in Ab(Set^{\mathcal{O}(X)^{op}})$, then $G \in Set^{\mathcal{O}(X)^{op}}$ and we have $m, i,$ and $e$ as in the definition of an abelian group object. For every $U \in \mathcal{O}(X)$ we consider $m_U, i_U,$ and $e_U$ such that the diagrammatic rules still hold, then, $G(U)$ is an abelian group, i.e., $G$ is a functor of $\mathcal{O}(X)^{op}$ to $Ab$. These correspondences  describe  an equivalence of categories  $Ab(Set^{\mathcal{O}(X)^{op}}) \simeq Ab^{\mathcal{O}(X)^{op}}.$ 

Observe that $Ab(Set) \simeq Ab$ and consider $E: Ab(Set) \rightarrow Set$ the forgetful functor ($E$ ``forgets'' the group operations); note that this functor preserves all limits. Thus an abelian sheaf is a functor $F: \mathcal{O}(X)^{op} \rightarrow Ab$ where the composition $\mathcal{O}(X)^{op} \rightarrow  Ab \rightarrow Sets$ is a sheaf of sets. Denote the category of abelian sheaves by $Sh_{Ab}(X)$. Since we have inclusions $Sh(X) \rightarrow Set^{\mathcal{O}(X)^{op}}$ and $Sh_{Ab}(X) \rightarrow Ab^{\mathcal{O}(X)^{op}}$, the equivalence $Ab(Set^{\mathcal{O}(X)^{op}}) \simeq Ab^{\mathcal{O}(X)^{op}}$ induces an equivalence $Ab(Sh(X)) \simeq Sh_{Ab}(X)$, since the subcategories of sheaves, over $Set$ and over $Ab$, are closed under all small limits.

Therefore, to apply cohomological techniques in $Ab(Sh(X))$ is equivalent to apply it in $Sh_{Ab}(X)$. Many classical books of Sheaf Cohomology prove that $Sh_{Ab}(X)$ is an abelian category (see, for instance, \cite[Theorem 2.5]{iversen1986cohomology}). We, alternatively,  can show that $Ab(\mathcal{E})$ is an abelian category for any topos $\mathcal{E}$ so, in particular, $Ab(Sh(X))$ is abelian. We will comment more on this in Section  \ref{sec:12}.

We will use right derived functors to define the cohomology group of sheaves, thus we need to ensure that $Sh_{Ab}(X)$ has enough injectives: see \cite[Theorem 3.1]{iversen1986cohomology} for a proof of this fact.

For every sheaf $F$ in $Sh_{Ab}(X)$ and $U$ open set of $X$, we have the abelian group of \textit{sections of $F$ over $U$} defined by $\Gamma(U,F)=F(U)$. Sections over $X$ are called \textit{global sections}, and $\Gamma(X,-):Sh_{Ab}(X) \rightarrow Ab$ is a left exact functor\footnote{It preserves all small limits.} that sends an abelian sheaf to its global section abelian group, know as \textit{global section functor}. 

Then  the $q$-group cohomology group of $X$ with coefficients in $F$  is, by definition, the $q$-th right derived functor of $\Gamma(X,F)$. In other words, given an injective resolution $F  \rightarrow I^{\bullet}$, we have $H^q(X,F) = R^q\Gamma (X,I^{\bullet})$.

A special type of sheaves are the flabby sheaves. As we will see, they are important because, like injective objects, they allow the construction of acyclic resolutions. By definition, if the restriction maps $s_U: \mathcal{F}(X) \rightarrow \mathcal{F}(U)$ is onto for every $U \subseteq X$ open, the sheaf $\mathcal{F}$ is flabby. Equivalently, $\mathcal{F}$ is flabby if $\mathcal{F}(V) \rightarrow \mathcal{F}(U)$  is onto for any pair $U \subseteq V$ of open sets in $X.$

\begin{prop}
Every injective sheaf is flabby.
\end{prop}
\begin{proof} To establish this result, we will need an auxiliary construction. 

Consider a functor $x_*: Set \rightarrow Sh(X)$, such that
  \[
    (x_*H)(U)= 
\begin{cases}
    H, & x \in U\\
    \{*\},              &  x \notin U 
\end{cases}
\]
where $H$ is set, $U$ an open set in $X$, and \{*\} unitary set. This is known as the \textit{skyscraper sheaf}. In the abelian sheaf version, we have $x_*: Ab \rightarrow Sh_{Ab}(X),$ $H \mapsto (x_*H)(U),$ with the difference $H$ is now an abelian group and $x_*H$ is a functor that sends open sets of $X$ to $H$ or in the trivial group.

For each $x \in X$, let $D_x$ be an injective abelian group. We define an injective sheaf $D := \prod\limits_{x \in X} x_{*}D_{x}$. It is not difficult to see that $D(X) \rightarrow D(U)$ is surjective, i.e, $D$ is flabby.

Now suppose $F$ is an injective sheaf. We will show that $F$ is flabby. Since $F$ is injective, for each $x \in X$, the stalk $F_x$ is an injective abelian group. Consider the family of injective abelian groups $D(F)_x := F_x, x \in X$. Then $D(F) := \prod\limits_{x \in X} x_{*}D(F)_{x}$ is an injective and flabby sheaf and, since $F(U) \to \prod\limits_{x \in U} F_{x}$ is a monomorphism, $U \in {\cal O}(X)$, there is a mono $i: F \rightarrow D(F)$. Since $F$ is an injective sheaf, we can select a morphism $f: D(F) \rightarrow F$ such that $f \circ i = id_F$. Since all components of identity morphism are surjective homomorphism, the same holds for the components of $f$.
Besides that, by naturality of $f$, the following diagram commutes:
\begin{center}
    \begin{tikzcd}
D(X) \arrow[r, "{s_{U,D}}"] \arrow[d, "f(X)"'] & D(U) \arrow[d, "f(U)"] \\
F(X) \arrow[r, "{s_{U,F}}"'] & F(U)
\end{tikzcd}
\end{center}

We already know $f(U)$ and ${s_{U,D}}$ are surjectives, so $f(U)\circ {s_{U,D}}$ is surjective. By commutative of the diagram, ${s_{U,F}}\circ f(X)$ is surjective and so also is ${s_{U,F}}$. This holds for every open set $U$ of $X$, then $F$ is a flabby sheaf. 
\end{proof}

Now we show that flabby sheaves can build acyclic resolutions.

\begin{prop}
If $F$ is an flabby sheaf, then $H^q(X,F) = 0,$ for all $q > 0$. In other words, $F$ is $\Gamma(X,-)$-acyclic.
\end{prop}

\begin{proof}Since $F$ is flabby, we can construct $0 \rightarrow F \xrightarrow{f} G \xrightarrow{g} Q \rightarrow 0$ an exact sequence, where $G$ is injective because $Sh_{Ab}(X)$ has enough injectives. By the proposition above, $G$ is flabby.

Using the left exactness of the global section functor, we immediately obtain the exact sequence  $$0 \rightarrow \Gamma(X,F) \xrightarrow{\Gamma_f} \Gamma(X,G) \xrightarrow{\Gamma_g} \Gamma(X,Q) $$

The flabby condition of $F$ implies more:  $$0 \rightarrow \Gamma(X,F) \xrightarrow{\Gamma_f} \Gamma(X,G) \xrightarrow{\Gamma_g} \Gamma(X,Q) \rightarrow 0 $$ is exact. This is not straightforward and uses Zorn's Lemma to be proved \cite[Theorem 3.5]{iversen1986cohomology}.

By Theorem \ref{theo:derivedfunctor}, the derived functors induce a long exact sequence. We will analyze the following part of the sequence: 

\begin{center}
    $\Gamma(X,G) \xrightarrow{\Gamma_{g_0}} \Gamma(X,Q) \xrightarrow{\delta_0} H^1(X,F) \xrightarrow{f_1} H^1(X,G)$ 
\end{center}

Where $g_0 = g$. Note $H^1(X,G)=0,$ because $G$ is injective. Since the sequence above is exact, by the Isomorphism Theorem:  $$H^1(X,F) \cong \frac{\Gamma(X,Q)}{Ker(g_0)}$$ But $g_0 = g$ is a surjective morphism, so $Ker(\delta) = Im(g_0) \cong \Gamma(X,Q).$ Then, $H^1(X,F) = 0$.

To conclude the result, use an induction argument in $q$ and the fact that if the first two objects in a short exact sequence are flabby, the third one is also flabby. 
\end{proof}

\textbf{Remark:} All proofs we know of this proposition require Zorn Lemma, so a constructive proof may not be available yet (or maybe there is not a constructive proof). 

Given the fact that every sheaf admits a flabby resolution, via Godement resolution, the Proposition above implies we can define cohomology groups with coefficient in $\mathcal{F}$ using flabby sheaves instead of injective ones. The reason why this is possible is that we need a procedure that measures the ``failure of its right exactness'' to construct cohomology, and the proposition above guarantees such procedure for flabby sheaves \cite{godement1958topologie}.

\subsection{\v{C}ech Cohomology}
\label{sec:7}

The nerve construction of an open covering first appeared in \cite{alexandroff1928allgemeinen}, before its debut in Sheaf Theory. Originally, the nerve associated an open covering of a topological space to an abstract simplicial complex, in an algorithmic form. Currently, nerve constructions preserve the algorithmic form but they deal with more general settings than topological spaces and simplicial complexes. We will use the \v{C}ech nerve to develop \v{C}ech Cohomology.

Godement improved in his book the brief discussion about \v{C}ech Cohomology made in Cartan's seminars, and it is a fundamental reference on the subject  until today. Additionally, we recommend Kozsul's note classes \cite{koszul1957faisceaux} and, for references in  English, there are algebraic geometry books  as \cite{hartshorne1977algebraic}. Here we introduce \v{C}ech Cohomology as a technique to calculate Sheaf Cohomology by taking open covers of a fixed topological space, construct a cochain complex from it, and finally compute the cohomology groups. We aim to use this section to compare it with \v{C}ech Cohomology for Grothendieck Toposes.

Fix $F$ in $Sh_{Ab}(X)$ and consider ${\cal U} = (U_i )_{i \in I}$ an open cover of $X$,  where $I$ is a  set of indices. For each $q \in \mathds{N}$, denote the $U_{{i_0},...,{i_q}} = U_{i_0} \cap ... \cap U_{i_q}$ for $i_0, ..., i_q \in I$ (this is the \v{C}ech nerve). The \textit{\v{C}ech cochain complex} is $$C^q({\cal U}, F) = \prod\limits_{i_0,...,i_q}F(U_{{i_0},...,{i_q}}),  \forall q \geq 0,$$ and its coboundary morphisms $d^q : C^{q}({\cal U},F) \to C^{q+1}({\cal U},F)$ are $$(d^q\alpha) = \sum\limits_{k=0}^{q+1}(-1)^k\alpha(\delta_k)_{\big|_{U_{{i_0},...,{i_{q+1}}}}}$$ where $\delta_k$ is used to indicate that we are removing $i_k$, i.e., $\alpha(\delta_k) = \alpha_{i_0,...,\widehat{i_k},...,i_{q+1}}$.

A straightforward verification shows that $d^{q+1}\circ d^q = 0$ so, indeed, this is a cochain complex and we can define the $q$-th \v{C}ech cohomology group of $F$ with respect to the covering ${\cal U}$ by $\check{\mathrm{H}}^q({\cal U},F) = Ker (d^{q})/Im (d^{q-1})$.

The result below gives us a first clue that \v{C}ech cohomology can be useful to calculate cohomology of sheaves. See \cite[Lemma III 4.4]{hartshorne1977algebraic} for a proof.

\begin{prop} \label{Cech-U}
Let $F$ be a sheaf in $Sh_{Ab}(X)$, and ${\cal U} = (U_i)_{i \in I}$ a covering of $X$. There is a canonical  morphism $k^q_{\cal U} : \check{\mathrm{H}}^q({\cal U},F) \rightarrow \mathrm{H}^q(X,F)$ natural and functorial in $F$ for each $q \in \mathds{N}.$ 
\end{prop}

Next, we will briefly examine the behavior of the \v{C}ech cohomology groups under the dynamic of refinements of coverings. We will return to this point later, in Section \ref{sec:13}.

Let ${\cal{V}} = (V_j)_{j \in J}$ be another covering of $X$. Suppose that ${\cal{U}}$ is a refinement of ${\cal{V}}$, i.e, for each $i \in I$, there is $j \in J$ such that $U_i \subseteq V_j$. Choose any function $c : I \to J$ such that $U_i \subseteq V_{c(i)}, i \in I$; then there is a induced morphism of cochain complexes $m_c: C^\bullet({\cal{V}}, F) \to  C^\bullet({\cal{U}}, F)$ and a corresponding morphism of \v{C}ech cohomology groups w.r.t. the coverings ${\cal U}$ and ${\cal V}$, $\check{m}_c : \check{\mathrm{H}}^\bullet({\cal V},F)  \to \check{\mathrm{H}}^\bullet({\cal U},F) $. Moreover, if $d : I \to J$ is another chosen function w.r.t. the refinement of ${\cal V}$ by ${\cal U}$, then the induced morphisms of complexes $m_c, m_d$ are homotopic, thus, by Proposition \ref{homotopic}, there is a unique induced morphism of cohomology groups $\check{m}_{{\cal U}, {\cal V}} : \check{\mathrm{H}}^\bullet({\cal V},F)  \to \check{\mathrm{H}}^\bullet({\cal U},F) $. 

Note that the class $Ref(X)$ of all coverings of $X$ is partially ordered under the refinement relation; this is a directed ordering relation.

The construction above is functorial in the following sense: 
\begin{itemize}
    \item $\check{m}_{{\cal U}, {\cal U}} = id : \check{\mathrm{H}}^\bullet({\cal U},F)  \to \check{\mathrm{H}}^\bullet({\cal U},F)$;
    \item  If ${\cal W} = (W_k)_{k \in K}$ is a covering of $X$ such that ${\cal V}$ is a refinement of ${\cal W}$, then $\check{m}_{{\cal U}, {\cal W}} = \check{m}_{{\cal U}, {\cal V}}  \circ \check{m}_{{\cal V}, {\cal W}} : \check{\mathrm{H}}^\bullet({\cal W},F)  \to \check{\mathrm{H}}^\bullet({\cal U},F)  $. 
\end{itemize}

The (absolute) \v{C}ech cohomology group is, by definition, the directed (co)limit\footnote{This (co)limit has to be taken with some set-theoretical care, we  will not  detail this point here.} $$\check{\mathrm{H}}^\bullet(X,F) :=   \varinjlim\limits_{{\cal U} \in Ref(X)}\check{\mathrm{H}}^\bullet({\cal U},F) .$$

The main result concerning \v{C}ech cohomology is the following:

\begin{teo} \label{Cech-te}

The canonical morphisms  $k^q_{\cal U} : \check{\mathrm{H}}^q({\cal U},F) \rightarrow \mathrm{H}^q(X,F), q \in \mathds{N},$ according notation in Proposition \ref{Cech-U}, are compatible under refinement. Moreover, the induced morphism on colimit 

$$k^q : \check{\mathrm{H}}^q(X,F) \to \mathrm{H}^q(X,F), q \in \mathds{N}, $$

is an isomorphism if $q  \leq 1$ and a monomorphism if $q=2$.

\end{teo}

Far more interesting, under reasonable geometrical hypothesis on the topological space $X$ (for instance, if $X$ is a Hausdorff paracompact space\footnote{This holds for any CW-complex or any topological manifold.}), then  the canonical morphisms $k^q$ are isomorphisms for all $q \geq 0.$

\subsection{Applications}
\label{sec:8}

Most of mathematicians will not be interested in abstract sheaf theory alone, but in its applications for specific sheaves. For example, if $(X,\mathcal{O}_X)$ is a ringed space, i.e., $X$ is a topological space and $\mathcal{O}_X$ is a ring-valued sheaf,  we can define a coherent sheaf $F$ on $(X,\mathcal{O}_X)$ that will look like a vector bundle with the advantage of forming an abelian category. Thus, we can study coherent sheaf cohomology. In this context, we have an analog of Poincaré Duality of Algebraic Topology, and the Serre Duality, that relates cohomology groups at level $n-q$ with $Ext$ groups at level $q$, where $n$ is the dimension of the particular scheme we are studying, by \cite[Theorem III 7.6]{hartshorne1977algebraic}). Coherent sheaf cohomology also provides a characterization of Euler Characteristic by an alternating sum of the dimension of cohomology groups of a scheme with coefficient in a coherent sheaf. 

We observe that schemes are essential in modern Algebraic Geometry, and its definition arises from the affine (locally) ringed space $(Spec(R),\mathcal{O}_R)$. We use Zariski Topology to construct the sheaf $\mathcal{O}_R$ and furnish the spectrum $Spec(R)$ of a commutative ring with a topological structure. The gluing of ringed spaces of the form $(Spec(R),\mathcal{O}_R)$ results in the notion of schemes. We may use schemes to construct quasi-coherent sheaves, a generalizatin of coherent sheaves introduced by Serre in \cite{10.2307/1969915}. Quasi-coherent sheaves constitutes an interesting class of coefficients for cohomologies in Algebraic Geometry: \v{C}ech and sheaf cohomology agree on a noetherian separated scheme with the Zariski topology, for any quasi-coherent sheaf as coefficient.

Another application of sheaf-theoretical methods is the relation between \v{C}ech cohomology and De Rham cohomology which is obtained as follows: Given a topological space $X$, and a set $A$,  the constant presheaf with values in  $A$ that we mentioned earlier can be transformed into a constant sheaf with values in $A$ by a standard ``sheafification” process.
In particular, the set $A$ can be the underlying set of an abelian group such as $\mathds{R}$, the additive group of real numbers, and the topological space can be a compact manifold $M$ of dimension $m$ and class at least ${\cal C}^{m+1}$. In this case, there is an isomorphism $H^q_{dR}(M) \cong \check{H}^q(M,\mathds{R})$, for all $q \leq m$, where $H^q_{dR}$ denotes the de Rham cohomology groups \cite[Appendix]{petersen2006riemannian}. Similarly, \v{C}ech cohomology and singular cohomology  coincide for any topological space $X$ that is homotopically equivalent to a CW-complex, with the constant sheaf of an abelian group $A$ as coefficient.

More recently, sheaf and \v{C}ech cohomologies have been used in quantum mechanics because of the general idea of measuring the obstruction between local and global properties. For example, in \cite{abramsky2015contextuality}, the \v{C}ech cohomology groups are defined for specifics topological spaces, with a corresponding open cover, and show they identify the obstructions that characterize logical forms of \textit{contextuality}.

In the next section, we will generalize the categories of sheaves over some topological space defining the notion of Grothendieck topos and exhibit specific Grothendieck topos that appears in other areas of Mathematics.

\section{Toposes}
\label{sec:9}

\subsection{Grothedieck Toposes}
\label{sec:10}

Cohomology groups often provide good invariants to classify objects: if two Riemann surfaces (with some additional conditions) agree in each level of the cohomology groups, then they are the same from a topological point of view. In the 1950s, this problem was well understood for algebraic curves over the field of complex numbers, but not much was known for algebraic curves over other fields. In 1954, Jean-Pierre Serre introduced sheaf theory in Algebraic Geometry with coherent sheaves \cite{10.2307/1969915}, and one year later, in \cite{serre1956geometrie}, he showed that with coherent sheaves in hand there are cases such that the cohomology groups of complex and non-complex algebraic varieties coincide, by using the Zariski topology.

However, in most cases, the Zariski topology does not have ``enough'' open sets. So, motivated to prove the Weil's Conjectures, A. Grothendieck had the idea of stop trying to find open sets, in the usual sense, and defined an analogous version of inclusion of open sets using more general morphisms in small categories. This gave birth to Grothendieck topologies and to Grothendieck toposes, particularly, the étale topos of a scheme X - the category of all étale sheaves on a scheme X - and so to Étale Cohomology. A. Grothendieck, M. Artin, and J-L. Verdier proved three of the four Weil's Conjectures, and the remaining one was proved by Deligne in 1974 \cite{deligne1974conjecture}. The main references to see the development of this program aiming the proof of Weil's Conjectures passes through Bourbaki seminars \cite{grothendieck1962fondements}, ``Eléments de Geométrie Algébrique'' \cite{PMIHES_1960__4__5_0,PMIHES_1961__8__5_0,PMIHES_1961__11__5_0,PMIHES_1963__17__5_0,PMIHES_1964__20__5_0,PMIHES_1965__24__5_0,PMIHES_1966__28__5_0,PMIHES_1967__32__5_0}, and ``Séminaire de Géométrie Algébrique'' (SGA). We highlight SGA4 \cite{grothendieck1972topos}, as the one dedicated to topos theory and étale cohomology.

Now, remember that a \textit{locale} $(L,\leq)$ is a complete lattice such that 
    \begin{center}
\begin{center}
      $a \wedge (\bigvee\limits_{i \in I} b_i) = \bigvee\limits_{i \in I}(a \wedge b_i)$, $ \forall a, b_i \in L$.
     \end{center}
     \end{center}
The poset of all open sets of a topological space $X$ is a locale. Locales coincide with complete Heyting algebras\footnote{The class of all  Heyting algebras provides the  natural algebraic semantics for the intuitionistic propositional logic, that is the ``constructive fragment'' of the classical propositional logic.}. 

Note that in the definition of a sheaf over a topological space  we did not use the points of the space, that is, only their locale structure was necessary. In fact, we can define sheaves for a presheaf $F: {\mathcal{L}}^{op} \rightarrow Set$, where $\mathcal{L}$ is the category associated to a locale $L$, since it is a poset. This is one simple case where the notion of sheaf is still available in a category different from $\mathcal{O}(X)$. There are others? Yes, by introducing an abstract idea of open cover we can define sheaves for any small category $\mathcal{C}$.

First, we will be a bit less general. Suppose $\mathcal{C}$ is a small category with finite limits (or just with pullbacks). A Grothendieck pretopology on $\mathcal{C}$ associates to each object $U$ of $\mathcal{C}$ a set $P(U)$ of families of morphisms $\{U_i \rightarrow U\}_{i \in I}$ satisfying some simple  rules. They are:
\begin{enumerate}
\item The singleton family $\{U' \xrightarrow{f} U \}$, formed by an isomorphism $f : U' \overset{\cong}\to U$,  is in $P(U)$;

\item If $\{U_i \xrightarrow{f_i} U\}_{i \in I}$ is in $P(U)$ and $\{V_{ij} \xrightarrow{g_{ij}} U_i \}_{j \in J_i}$ is in $P(U_i)$ for all $i \in I$, then $\{V_{ij} \xrightarrow{f_i \circ g_{ij}} U \}_{i \in I, j \in J_i}$ is in $P(U)$; 
\item If $\{U_i \rightarrow U\}_{i \in I}$ is in $P(U)$, and $V \rightarrow U$ is any  morphism in $\mathcal{C}$, then the family of pullbacks $\{V \times_U U \rightarrow V \}$ is in $P(V)$.

\end{enumerate}
 
The families in $P(U)$ are called \textit{covering families} of $U$.\\

\textbf{Example:}  Note that the ``concrete'' notion of cover of topological spaces provides an example of Grotendieck pretopology: an object in $\mathcal{O}(X)$ is an open set $U$ in $X$ and  the morphisms in $\mathcal{O}(X)$ are inclusions of open subsets of $X$, this category has all finite limits (given by finite intersection of open subsets).  Thus is natural to define a Grothendieck pretopology $P$ in $\mathcal{O}(X)$ by
$$\{U_i \overset{f_i}\hookrightarrow U\}_{i \in I} \in P(U) \iff U = \bigcup\limits_{i \in I} U_i.$$ This can be carried out analogously for any locale ${\cal L}$.

We say that the presheaf $F: \mathcal{C}^{op} \rightarrow Set$ is a sheaf for the Grothendieck pretopology $P$ if the following diagram is an equalizer in $Set$:  

             \begin{center}
              \begin{tikzcd}
            F (U) \arrow[r] & \prod\limits_{i \in I} F (U_i) \arrow[r, shift left=1 ex] 
            \arrow[r, shift right=0.5 ex]  & \prod\limits_{(i,j)\in I\times I} F (U_i \times_U U_j)
            \end{tikzcd}
            \end{center} \label{grothsheaf}
            
However, different pretopologies can provide the same class of sheaves. For instance, if $\bigcup_{i \in I} U_i = U$ is an open cover of the open subset $U \subseteq X$, for any $V \subseteq U_j$, for some $j \in I$, we have $ V \cup \bigcup_{i \in I} U_i = U. $ To remove this ambiguity from the above definition we use the  notion of  {\em covering sieve}. 

Let $C$ be an object in a small category $\mathcal{C}$ (the assumption of existence of the pullbacks over ${\cal C}$ can be dropped now), a \textit{sieve} on $C$ is a collection $S$ of morphisms $f$ with codomain $C$ such that $f \circ g \in S,$ for all morphism $g$ with $dom(f) = cod(g)$. Given $h: D \rightarrow C$, define $$h^{\ast}(S) = \{ g \mid cod(g) = D, \ h\circ g \in S\}$$ The $h^{\ast}(S)$ will assume the role of a pullback in the category, as we see bellow: 

A \textit{Grothendieck Topology} in $\mathcal{C}$  associates each object $C$ of  $\mathcal{C}$ to a collection $J(C)$ of sieves on $C$ such that:
			\begin{enumerate}
				\item The maximal sieve on $C$, $\{f \mid cod(f) = C\}$, is in $J(C)$;
				\item If $R$ and $S$ are sieves on $C$, $S$  is in $J(C)$ and  $h^{\ast}(R)$ is in $J(D)$  for all $h: D \rightarrow C$ in $S$, then $R$ is in $J(C)$;
				\item If $S$ is in $J(C)$, then $h^{\ast}(S)$ is in $J(D)$ for all $h: D \rightarrow C$.
		    \end{enumerate}	
		    
	The collection of sieves in $J(C)$ are the \textit{covering sieves} (or $J$-covers). The pair $(\mathcal{C},J)$ formed by a small category $\mathcal{C}$ and a Grothendieck Topology $J$ is called \textit{site}. 	Each pretopology $P$ on a category with pullbacks ${\mathcal{C}}$ determines a least  Grothendieck topology $J_P$ on $\mathcal{C}$: a covering sieve $S \in J_P(U)$ is a sieve on the object $U$ that contains some family in $P(U)$. 
	
	We can also define sheaves for the Grothendieck topology $J$, but more concepts would be introduced and we can be satisfied with what we have since both definitions - for Grothendieck topologies and pretopologies - are equivalent\cite{johnstone77topostheory}. In particular, if ${\mathcal{C}}$ is a small  category with pullbacks, a presheaf $F : \mathcal{C}^{op} \to Set$ is a sheaf for the pretopology $P$ iff it is a sheaf for the induced topology $J_P$. Morphisms of sheaves are natural transformations, and so we obtain $Sh(\mathcal{C},J),$ the category of sheaves over this site.
			 

    Finally, a \textit{Grothendieck Topos} is a category that is equivalent to $Sh(\mathcal{C},J)$, for some site. Note $Sh(X) = Sh(\mathcal{C},J_P)$ is a Grothendieck topos where $\mathcal{C} = \mathcal{O}(X)$ and $J_P$ is the Grothendieck topology generated by the pretopology $P$ described in the example above (that pretopology is not a topology).
    
   Grothendieck toposes also are characterized by purely categorical axioms, by Giraud’s Theorem \cite[Theorem 0.45]{johnstone77topostheory}. If a category has some specific properties, it is a Grothendieck topos. Conversely,  every Grothendieck topos satisfies these same properties. 
   
   We provide below a list of properties we will need in Section \ref{sec:12} to sketch the proof that $Ab(\mathcal{E})$ is $AB5$ and has generators:

\begin{lem}\label{giraud}
	A Grothendieck topos $\mathcal{E}$ satisfies the following conditions:
		\begin{enumerate}
			\item all colimits are universal (i.e, preserved by pullback);
			\item has all small coproducts;
			\item has a set of generators (i.e., exists a small family $\{G_i\}_{i \in I}$ of objects in $\mathcal{E}$ where given distinct morphisms $f,g: X \rightarrow Y$ in $\mathcal{E}$ there are $i \in I$ and $h: G_i \rightarrow X$ such that $f \circ h \neq g \circ h$);
			\item filtered colimits commute with finite limits.

		\end{enumerate}
\end{lem} 

In this list, only the last property is not part of Giraud’s Theorem, but we will use it and it follows, not immediately, from the fact the same holds for $Set$.

\subsection{Elementary Toposes}
\label{sec:11}

We add here a short section on a generalization of Grothendieck topos: the categorical concept of ``elementary topos'', introduced by Lawvere and Tierney in the early 1970s. The relatively simple axioms that defines an elementary topos allows a description of an internal language and an internal (intuitionistic) logic: this machinery is useful to perform ``high level arguments'', for instance to  provide a simple proof that the category of abelian group objects in an elementary topos is an abelian category (Theorem \ref{abE-th}). This ``high-level'' method was successfully explored in Algebraic Geometric \cite{blechschmidt2018using}, revealing its potential to other applications in Mathematics;  first steps towards high-level Homological Algebra were given in \cite{blechschmidt2018flabby}.

An \textit{elementary topos} is a (locally small) category that is cartesian closed, has a subobject classifier, and has all finite limits (or equivalently, has all finite products and equalizers, or has pullbacks and a terminal object). 

A category is cartesian closed if it has binary products and it is possible to define an exponential object for every two objects as follows: given $B$ and $C$ objects, there is an \textit{exponential object} $C^B$ endowed with an \textit{evaluation map} $ev: C^B \times B \rightarrow C$ such that for any other object $A$, endowed with an arrow $f: A \times B \rightarrow C$, there is a unique morphism $\Bar{f}: A \rightarrow C^B$ where $ev \circ (\Bar{f}\times id_B) = f$. An important property that arises from this definition is the isomorphism $\phi_{AC}^B: Hom(A \times B, C) \overset\cong\to Hom(A, C^B)$, which are natural in $A$ and $C$.

A subobject classifier of a locally small category that has all finite limits, and $1$ as terminal object consists of an object $\Omega$ of \textit{truth values} and a \textit{truth morphism}\footnote{In fact, these data are unique up to unique isomorphisms.} $t: 1 \rightarrow \Omega$ such that given any object $E$, and any ``subobject'' \begin{tikzcd}
r: U \arrow[r, tail] & E
\end{tikzcd}, there is a unique morphism \begin{tikzcd}
\chi_r: E \arrow[r] & \Omega
\end{tikzcd} that makes the following diagram a pullback:
 \begin{center}
    \begin{tikzcd}
    U \arrow[d, "r"', tail] \arrow[r, "!"] & 1 \arrow[d, "t"] \\
    E \arrow[r, "\chi_r"'] & \Omega
    \end{tikzcd}
    \end{center}

This \begin{tikzcd}
\chi_r: E \arrow[r] & \Omega
\end{tikzcd} is called \textit{characteristic morphism of $r$}. It can look too abstract, but when the category is $Set$, we have $\Omega = \{0,1\}$ and, for each subset $U$ of a fixed set $E$, the morphism $\chi_U$ is the well known characteristic function. In fact, $Set$ is an example of elementary topos \cite[Example 5.2.1]{borceux1994handbook3}. More generally, 
every Grothendieck topos is an elementary topos \cite[Example 5.2.9]{borceux1994handbook3}. 

Any elementary topos ${\cal E}$ enjoys some categorical properties that holds in the category $Set$, e.g.: a morphism in ${\cal E}$ is an isomorphism iff it a monomorphism and an epimorphism; every epimorphism in ${\cal E}$ is a coequalizer; any morphism in ${\cal E}$ has a (unique up to unique isomorphism) factorization through the image - it is composition of a monomorphism with an epimorphism. 

An important type of morphism between toposes $f: \mathcal{F} \rightarrow \mathcal{E}$ is called \textit{geometric morphism}. It consists of a pair of functors, $f_* : \mathcal{F} \rightarrow \mathcal{E}$,  the \textit{direct image}, and $f^*: \mathcal{E} \rightarrow \mathcal{F}$, \textit{the inverse image}, such that:
\begin{enumerate}
    \item $f^*$ is left adjoint of $f_*$;
    \item $f^*$ preserves finite limits, i.e, it is left exact.
\end{enumerate}

The reader does not need to know the definition of adjoint pair of functors to understand the ideas covered in this survey and can think in adjointness as an abstraction of free constructions in Algebra; the sheafification process is an instance of adjointness; the ``exponencial convertion'', described by the natural isomorphisms $\phi^{B}_{AC}$ above shows that the functor $(-) \times B$ is the left adjoint of the functor $( - )^B$. We recommend \cite{zbMATH01216133} if there is an interest to better understand the  proofs in which we explicitly use the adjoint property of geometric morphisms.

In general, each side of an adjoint pair of functors determines the other side, up to isomorphism, so it is clear that $Set$ is the terminal Grothendieck topos, concerning the geometric morphisms. This motivates the definition of point in a topos ${\cal E}$; it is a geometric morphism $f : Set \to {\cal E}$.

Another distinguished geometric morphism is  $i = (i_*, i^*) : Sh({\cal C},J) \to Set^{{\cal C}^{op}}$, here the direct image part is just the (full) inclusion $i_* : Sh({\cal C},J) \hookrightarrow Set^{{\cal C}^{op}}$ and the inverse image part is the sheafification functor, $i^* :  Set^{{\cal C}^{op}} \to  Sh({\cal C},J)$.

Every topos naturally encodes a ``local set theory'' \cite{bell1988topos}. Indeed, each topos has an internal language, known as \textit{Mitchell-Bénabou language}, and a canonical \textit{interpretation} - a procedure to give a meaning for the symbols introduced in the canonical language. In the next section, we will use these notions to proof Theorem \ref{abE-th}.

Provide the complete definition of the internal language of a topos and its respective interpretation would spend about tree pages of this survey so we have restricted ourselves to only present a general idea. We hope this approach helps to understand the rigorous definitions given in \cite[Chap 6]{borceux1994handbook3}.

Given a topos $\cal E$, the  Mitchell-B\'enabou language $L(\cal{E})$ consists of three parts:
\begin{center}
   $\bullet$ Sorts (or types)  $\enspace\enspace\enspace\enspace\enspace\bullet$ Terms  $\enspace\enspace\enspace\enspace\enspace\bullet$ Formulas 
\end{center}
For each object $A$ in ${\cal E}$, there is an associated sort $s_A$ (they are distinct from each other). The terms $\tau$ of $L({\cal E})$ have a value sort $s(\tau)$ and are inductively defined from the basic terms by applying certain natural constructors - the basic terms of sort $s_A$ are the constants of value sort $s_A$ that corresponds to a morphism  $1 \to A$ in ${\cal E}$ and an enumerable set of variables $\{x^A_i: i \in \mathds{N}\}$ of sort $s_A$; more complex terms, $t:s_B$, are inductively constructed from simpler terms $t_0:s_{A_0}, \cdots, t_{n-1}:s_{A_{n-1}}$ by a  formal application of morphisms $t = f(t_0, \cdots, t_{n-1})$, where $f : A_0 \times \cdots \times A_{n-1} \to B$ is an arrow in ${\cal E}$. Formulas are inductively constructed from the basic (or atomic) formulas by applying (firs-order and higher-order) logical constructors - an atomic formulas is defined ``to abbreviate relations between terms''.  As a simple  example of (atomic) formula we have $\tau =_A \sigma$, where $\tau$ and $\sigma$ are terms with same value sort $s_A$.

For the canonical interpretation of the language $L({\cal E})$ in the topos ${\cal E}$, the main idea is establish a \textit{realization} for each term, and a \textit{truth table} for each formula, as follows:
Let $\tau$ be a term of type $s_A$ with variables $x_1,...,x_n$ of types $s_{X_1},...,s_{X_n}$, respectively. A \textit{realization} of $\tau$ is an arrow in ${\cal E}$, written $[\tau]: X_1 \times ... \times X_n \rightarrow A $.

Now, given a formula $\varphi$ with (free) variables $x_1,...,x_n$ of types $s_{X_1},...,s_{X_n}$, a \textit{truth table} of $\varphi$ is an arrow in ${\cal E}$, $[\varphi]: X_1 \times ... \times X_n \rightarrow \Omega $, where $\Omega$ is the subobject classifier of ${\cal E}$. 

Next, we exemplify the abstract ideas presented above
\begin{enumerate}
    \item[(i)] Consider $x$ a variable of type $s_A$. Then its realization is established to be the identity morphism $$[x]  \overset{\mathclap{\strut\text{def}}}= id_A : A \to A$$
    \item[(ii)] Previously, we mentioned that $\tau=_A \sigma$, where $s(\tau) = s(\sigma) = s_A$, is a formula. Continuing this, we establish that its truth table is the morphism $$X_1 \times ... \times X_n \xrightarrow{([\tau], [\sigma]))} A \times A \xrightarrow{\delta_A} \Omega$$ where the free variables in $\tau$ and $\sigma$ have types among $s_{X_1},..., s_{X_n}$ and $\delta_A$ is the characteristic morphism of $\triangle_A \overset{\mathclap{\strut\text{def}}}= (id_A,id_A) : A \rightarrow A \times A$, the diagonal morphism.  
\end{enumerate}

In this setting, a formula $\varphi$ can be valid or not. We say $\varphi$ is \textit{valid} if the canonical  interpretation $$X_1 \times .... \times X_n \xrightarrow{!} 1 \xrightarrow{t} \Omega$$ is the truth table of $\varphi$, where $x_1,...,x_n$ are free variables of types $s_{X_1},...,s_{X_n}$. To denote that $\varphi$ is \textit{valid} we use  ${\cal E} \models \varphi$.

As a simple example, we will show that $\mathcal{E} \models x =_A x$, where $x$ is a variable of type $s_A$; in other words, the formula $x =_A x$ is valid in $\mathcal{E}$. By the discussion above, we know that the truth table of $x =_A x$ is $A \xrightarrow{([x], [x]))} A \times A \xrightarrow{\delta_A} \Omega$. Since $[x] = id_A: A \rightarrow A$, the morphism $([x],[x])$ is precisely the diagonal morphism $\triangle_A$. By definition, $\delta_A$ is the characteristic morphism of $\triangle_A$. Thus, the following diagram is a pullback
\begin{center}
    \begin{tikzcd}
A \arrow[r, "!"] \arrow[d, "{([x],[x])}"'] & 1 \arrow[d, "t"] \\
A\times A \arrow[r, "\delta_A"'] & \Omega
\end{tikzcd}
\end{center}

In particular, the diagram commutes, so $A \xrightarrow{!} 1 \xrightarrow{t} \Omega$, is the truth table of $x =_A x$. Therefore, $x =_A x$ is valid in $\mathcal{E}$.

We saw $x=_A x$ is valid for any topos, but the equality sign carries a lot of information - it is a specific characteristic morphism - and, at first, any property regarding $=_A$ should work explicitly with characteristic morphism and other tools presented by the internal language. However, when we work with toposes is usual to omit the internal language machinery and {\em pretend} that objects are sets, monomorphisms are injective functions, epimorphisms are surjective functions, isomorphisms are bijective functions, and so on. Basically, we {\em pretend} that a topos is the topos $Set$, which is possible due to the \textit{Soundness Theorem} \cite[Chap 15]{mclarty1992elementary}. This advantage has a cost: we only can replicate a construction in $Set$ to an arbitrary topos if we restrain ourselves to ``constructive aspects'' presented in intuitionistic logic, because, in general, the law of excluded middle ( i.e., $ \varphi \lor \neg \varphi$) does not hold for all toposes.  For the same reason, we avoid using the axiom of choice, a ``non-constructive'' set-theoretical axiom. 
We will apply this procedure in the next section (see Theorem \ref{abE-th}).

\subsection{Grothedieck Topos Cohomology}
\label{sec:12}
Now, we replicate the formerly cohomology constructions (Section \ref{sec:6}). We hope it is clear that the (Grothendieck) Topos Cohomology exhibited is an extension of Sheaf Cohomology; however, new techniques are necessary to prove the toposes versions of the results introduced in the previous section\footnote{On the other hand, if a Grothendieck topos has  ``enough points'', then its cohomology coincides with some  spatial sheaf cohomology, see  \cite{moerdijk1996cohomology}.}. 
We begin with a useful but simple concept:

A parallel pair of morphisms $f,g: A \rightarrow B$ is \textit{reflexive} if exists a common  section $s: B \rightarrow A$ of $f$ and $g$, that is, $f \circ s = id_B = g\circ s$. In particular, a reflexive coequalizer is a coequalizer of a reflexive pair.

\begin{lem} \label{abE-le}

\begin{enumerate}
    \item For any elementary topos $\cal{E}$, the forgetful functor \\ $E: Ab(\mathcal{E}) \rightarrow \mathcal{E}$ creates limits and reflexives coequalizers \cite[Section 6]{johnstone77topostheory};
    \item For abelian categories, the $AB5$ condition is equivalent to the category having all small colimits and all filtered colimits being universal  \cite{grothendieck1957}.
\end{enumerate}
\end{lem}

We use the above lemma to sketch the proofs of the main results in this subsection.

\begin{teo} \label{abE-th}
The category $Ab(\mathcal{E})$ is abelian for any elementary topos $\mathcal{E}.$
\end{teo}
\begin{proof}
Show that $Ab(\mathcal{E})$ is $Ab$-category follows by straightforward calculations. To see it is an additive category we will use the internal language of $\mathcal{E}$. Thus we need to prove that the following objects exist in $Ab(\mathcal{E})$: terminal and initial objects, binary products, and binary coproducts; moreover finitary products and finitary coproducts must coincide in $Ab(\mathcal{E})$.

We already know that terminal objects and binary products exists in $Ab(\mathcal{E})$ because $\mathcal{E}$ has finite limits and, by Lemma \ref{abE-le}.1, the forgetful functor creates finite limits, so $Ab(\mathcal{E})$ has finite limits.

We will use the internal logic of the topos to continue the proof. Suppose  $\mathcal{E} = Set,$ then $Ab(\mathcal{E}) \simeq Ab.$ It is know that $Ab$ is an additive category and the demonstration of this fact only uses constructive arguments. By the discussion in \ref{sec:11}, $Ab(\mathcal{E})$ is an additive category for any topos $\cal{E}$, and not only $\mathcal{E} = Set$. Usually, this argumentation is enough, but let's elaborate a bit more.

Consider $X$ an object of a topos $\cal{E}$ equipped with morphisms $m_X: X \times X \rightarrow X,$ $i_X: X \times X$ and $e_X: 1 \times X$ satisfying the following formulas of abelian groups 
\begin{center}
$\forall x : X, y: X, z: X \,(m_X(m_X(x,y),z) = m_X(x,(m_X(y,z))$\\ $\forall x: X, 0 : 1 \, (m_X(e_X \times id_X)(0,x)) = m_X(id_X \times e_X(x,0)))$ \\ $\forall x: X \, (m_X(i_X \times id_X(\triangle(x))) = e_X(!(x)))$\\
$\forall x : X, y: X \, (m_X(x,y) = m_X(y,x))$
\end{center}
This formulas may be described by diagrams, so $X$ is equivalent to an object in $Ab(\mathcal{E})$ (see Section \ref{sec:4} to remind the notation). In particular, the terminal object in $Set$ - a singleton - corresponds to the terminal object in $Ab(Set) \simeq Ab$ - the trivial group denoted by $1$ - thus, for each object $A$ of $Ab(Set)$, there is a unique morphism $f: A \to 1$. In other words: $$\forall A : Ab \,\, (f: A \to 1) \wedge (g:  A \to 1) \implies f=g$$ 
That is, the terminal object can be described by a formula. In the same way, the fact that there is a unique morphism $1 \to A$, for all $A$ object in $Ab$, also can be described by a formula. Therefore, we have a constructive proof that $1$ is a zero object for $Ab(Set)$. The \textit{Soundness Theorem} mentioned in the previous section guarantees we can replace $Set$ by any topos $\cal{E}$, so $1$ is a zero object in $Ab(\mathcal{E})$. 
Similarly, if we take the product of two objects in $Set$, we may obtain a product in $Ab(Set)$. Since in $Ab(Set)$ the product has an (internal) abelian group structure and all finitary products are finitary coproducts in $Ab$ (direct sum and direct product coincide for finite indices), we conclude the binary product is a coproduct in $Ab(Set)$ (in fact, they are biproducts, an equational notion described in Section 2.1) and, again, this proof is constructive thus is still valid in $Ab(\mathcal{E})$. Summing up,  $Ab(\cal{E})$ is an additive category for any elementary topos $\cal{E}$ see \cite[Chap 16.6]{mclarty1992elementary} for the description of the coproduct diagram using formulas.

It is not difficult to see that any  morphism $f: A \rightarrow B$ in $Ab({\cal E})$ has a kernel; since the forgetful functor creates limits (Lemma \ref{abE-le}.1),  $ker(f)$ is the equalizer $equal(0,f)$.

Now, let $f$ be an epimorphism in $Ab({\cal E})$, then it is also an epi in ${\cal F}$. But any epi in ${\cal E}$ is a coequalizer, then $f = coeq(g,h)$ for some $g, h \in {\cal E}$. Since $f \in Ab({\cal E})$, $f$ can be rewritten as $f=coeq(g',h')$ for some $g',h' \in Ab({\cal E})$.  Thus $$f = coeq(g',h') = coeq(0, h'-g') = coker(h'-g')$$  

To conclude $Ab(\mathcal{E})$ is an abelian category we have to construct a cokernel of an arbitrary morphism $f: A \rightarrow B$ in $Ab({\cal E})$ and show that any monomorphism in $Ab(\mathcal{E})$  is a kernel in $Ab(\mathcal{E})$.

We take a coequalizer in $\mathcal{E}$ of the pair $m \circ (f \times id_B)$ and $p_2$, where $m: B \times B \rightarrow B$ is the morphism $m$ introduced at the definition of group object, $p_2: A\times B \rightarrow B$ is the projection in the second coordinate, and $f: A \rightarrow B$ is a morphism in $Ab(\mathcal{E})$.

Let $q = coeq(m \circ (f \times id_B),p_2)$. First, note that \begin{tikzcd}
A \times B \arrow[r, "m\circ(f\times id_B)", shift left=1 ex] 
\arrow[r, "p_2"', shift right=0.5 ex] 
& B\end{tikzcd} is a reflexive pair with section $s = (0,id_B): B \rightarrow A \times B$. Considering parts of the diagram of coequalizer, and of the cartesian product of morphisms, we have:
\begin{center}
\begin{tikzcd}
B \arrow[r, "s"] & A\times B \arrow[d, "p_1"'] \arrow[r, "f\times id_B"] \arrow[rr, "p_2", bend left=49] & B\times B \arrow[r, "m"] \arrow[d, "p_1"] & B \arrow[r, "q"] & C \\
                 & A \arrow[r, "f"']                                                                     & B \arrow[ru, "id_B"']                     &                  &  
\end{tikzcd}
\end{center}

With a lot of diagram calculations and the coequalizer universal property, is possible to show that $q$, a coequalizer in $Ab(\mathcal{E})$, is the cokernel of $f,$ for any $f$ in $Ab(\mathcal{E}).$ 

Finally, let $f$ be a monomorphism in $Ab(\mathcal{E})$. Denote $coker(f) = q$, then $q\circ f =0$. Since $q \circ ker(q) =0$, by the universal property of $ker(q)$, there exists a unique $t \in Ab({\cal E})$  such that $f = ker(q) \circ t$ and this $t$ is a mono, because $f$ is a mono. Until now, all the information were obtained from very general categorical arguments. However, ${\cal E}$ is an elementary topos and we can simulate in ${\cal E}$ the proof, made in $Set$ with elements, that establishes that $t$ is ``surjective'' (i.e. an epimorphism) in ${\cal E}$ thus, as we already mentioned before, it follows that $t$ it is an isomorphism in the topos ${\cal E}$. Since $t \in Ab({\cal E})$, $t$ is an isomorphism in $Ab({\cal E})$. Summing up, we have shown that any mono in $Ab({\cal E})$ is a kernel, indeed, it is the kernel of its own cokernel.  

\end{proof} 

\textbf{Remark:} It is possible to prove the above Theorem without the internal language machinery, but the paper \cite{10.2307/43681686} shows it requires 10 pages of diagram calculations to fulfill the verification. We know that a proper introduction to the internal language of toposes requires more than 10 pages, yet it is more convenient and efficient in the long haul.

By the Grothendieck Theorem \ref{grotheorem}, if an abelian category satisfies $AB5$ and has a generator then it has enough injectives. Thus, to state $Ab(\mathcal{E})$ has enough injectives, we only need to prove this two conditions.

\begin{teo} 
If $\mathcal{E}$ is a Grothendieck topos, then $Ab(\mathcal{E})$ is $AB5$ category and has a generator.
\end{teo} 
\begin{proof}
Let's see that $Ab(\mathcal{E})$ satisfies $AB5.$ By Lemma \ref{abE-le}.2, we need to prove that $Ab(\mathcal{E})$ has all small colimits with all filtered colimits being universal. The first part - $Ab(\mathcal{E})$ has all small colimits - can be proven in several ways, none of them is simple: a possible form is presented in the proof of \cite[Theorem 8.11.iii]{johnstone77topostheory}; another argument follows from the construction of the associated sheaf functor (or sheafification functor), which is left adjoint to the inclusion functor $i: Sh(C,J) \to Set^{C^{op}}$ and preserves colimits. We choose not to present the associated sheaf functor here, but we indicate \cite[Chap III.5]{maclane1992sheaves} for a complete explanation. 

In a Grothendieck topos, filtered colimits and finite limits commutes. Since ${E}$ creates finite limits, it creates filtered colimits and pullbacks. Besides that, all colimits are universal in a topos, that is, they are preserved by pullbacks. Thus filtered colimits are universal in $Ab(\mathcal{E}).$ See \ref{giraud} to remember Grothendieck toposes's properties.

Now we prove that $Ab(\mathcal{E})$ has a set of generators.

By Giraud Theorem, $\mathcal{E}$ has a set of generators $\{G_i\}_{i \in I}$. Let $f,g: X \rightarrow Y$ in $Ab(\mathcal{E})$ so $f$ and $g$ are morphisms in $\mathcal{E}$. If $f \neq g$, since $\{G_i\}_{i \in I}$ is a generator of $\mathcal{E}$, there is $h_i: G_i \rightarrow E(X),$ for some $i \in I$, such that $$E(f) \circ h_i \neq E(g) \circ h_i$$
Consider the coproduct universal morphism $h:\coprod_{i \in I}G_i \rightarrow E(X)$ and the canonical morphism $\alpha_i: G_i \rightarrow \coprod_{i \in I}G_i$. We have $h_i = h \circ \alpha_i$ so
$$E(g) \circ h \circ \alpha_i = E(g) \circ h_i \neq E(f) \circ h_i =  E(f) \circ h \circ \alpha_i $$

Then $E(g) \circ h \neq E(f) \circ h$.

Use the fact that the forgetful functor has a left adjoint functor, $Z:  \mathcal{E} \rightarrow Ab(\mathcal{E})$; this is a generalization of the ``free abelian group'' construction from the topos $Set$ to any Grothendieck topos\footnote{ Because these toposes have the internal ``set of all natural numbers''.}. Apply it in $h$, so  $Z(h):Z(\coprod_{i \in I}G_i) \to X,$ is the associated morphism in $Ab({\cal E})$.
The adjointness of $Z$ and $E$ guarantees that $f \circ Z(h) \neq  g \circ Z(h)$, thus $Z(\coprod_{i \in I}G_i)$ is generator of $Ab(\mathcal{E}).$

\end{proof} 

Therefore, by the Grothendieck Theorem:
\begin{teo}
If $\mathcal{E}$ is a Grothendieck topos, then the abelian category $Ab(\mathcal{E})$ has enough injectives.
\end{teo}

 By Theorem \ref{theo:derivedfunctor}, we need a left exact additive functor, which we will again call global section functor. First, note there is a unique (up to isomorphism) geometric morphism $\Gamma : {\cal{E}} \to  Set$, and is enough to argue its inverse image $\Gamma^*: Set \to \cal{E}$ is unique (up to natural isomorphisms): by definition of geometric morhism $\Gamma^*$ send terminal objects to terminal objects, and preserver colimits. Besides that, every set is a disjoint union of singletons (terminal objects in $Set$), so $\Gamma^*$ must be given by $S \mapsto \coprod_{s \in S} 1,$  where $1$ is the terminal object in $\cal{E}$. 

Next, we take the direct image functor of $\Gamma$, which must be $\Gamma_* = Hom_{\mathcal{E}}(1,-): \mathcal{E} \to Set$. Finally, we define the global section functor as
$$\Gamma_{Ab} := Hom_{\mathcal{E}}(1,-): Ab(\mathcal{E}) \rightarrow Ab(Set),$$
Since $\Gamma_{Ab}$ is induced by the direct image of $\Gamma$ then, by Lemma \ref{abE-le}.(1), $\Gamma_{Ab}$ preserves limits, thus it is left exact.

Is not difficult to prove that the direct image of any geometric morphism preserves injectives. We will show this here to introduce a usual and simple manipulation with direct and inverse images using adjoint properties:

Let $f: \mathcal{F} \rightarrow \mathcal{E}$ geometric morphism, and $I$ injective object in $Ab(\mathcal{F})$. Consider the following diagram in $Ab({\cal E})$
\begin{center}
\begin{tikzcd}
X \arrow[r, "m", tail] \arrow[d, "h"] & Y \\
f_*(I) & 
\end{tikzcd}
\end{center}
The adjoint property of geometric morphisms allow us to transpose this diagram and obtain the following diagram in $Ab({\cal F})$
\begin{center}
\begin{tikzcd}
f^*(X) \arrow[r, "f^*(m)", tail] \arrow[d, "\tilde{h}"] & f^*(Y) \\
I & 
\end{tikzcd}
\end{center}
Note that $f^*(m)$ is a monomorphism in ${\cal F}$ and thus in $Ab({\cal F})$, since $f^*$ preserves finite limits.

Next, we use the injectiviness of $I$ to complete the diagram with $ g : f^*(Y) \to I$ that makes it commutative in $Ab({\cal F})$ . Then we transpose, by adjoint property, one last time, and find a commutative diagram in $Ab({\cal E})$ that guarantees that $f_*(I)$ is an injective object in $Ab(\mathcal{E})$
\begin{center}
\begin{tikzcd}
X \arrow[r] \arrow[d] & Y \arrow[ld] \\
f_*(I) & 
\end{tikzcd}
\end{center}

We define the $q$-th cohomology group of $\mathcal{E}$ with coefficientes in $F$, object in $Ab(\mathcal{E})$ as the $q$-th right derived functor of $\Gamma_{Ab}(F)$. In other words, $$H^q(\mathcal{E}, F) = R^q(\Gamma_{Ab})(F)$$

We can define cohomology for objects different from the terminal: Let $B$ be an object of $\mathcal{E}$, since $Hom_{\mathcal{E}}(B,-)$ is a left exact functor we can consider right derived functor for it, denoted by $H^q(\mathcal{E},B;F)$. The problem is how to describe $H^q(\mathcal{E},B;F)$ in terms of $H^q(\mathcal{E}, F)$. The idea is that the funtor $B^* : \mathcal{E} \rightarrow \mathcal{E} \downarrow B$, which sends an object $A$ in $\mathcal{E}$ into  $p_2: A \times B \rightarrow B$ in $\mathcal{E} \downarrow B$, induces an exact functor $B^*_{Ab} : Ab(\mathcal{E}) \rightarrow Ab(\mathcal{E} \downarrow B)$ that preserves injectives, and is possible to establish an isomorphism $H^q(\mathcal{E}, B;F) \cong H^q(\mathcal{E}\downarrow B,B^*_{Ab}(F))$. See \cite[page 262]{johnstone77topostheory}.

Grothendieck toposes admit a notion of flabby object: We say that $F$ in $Ab(\mathcal{E})$ is \textit{flabby} if $H^q(\mathcal{E}, B; F) = 0$, for all $q>0$ and all $B$ object in $\mathcal{E}$.

\begin{prop}
Every injective object in $Ab(\mathcal{E})$ is a flabby  object in $Ab(\mathcal{E})$.
\end{prop}
\begin{proof}

More generally, for any injective object $F$ in an abelian category we have an injective resolution $0 \rightarrow F \xrightarrow{id_F} F \rightarrow 0 \rightarrow 0 \rightarrow ...$ of $F$. Applying the left exact functor $\Gamma$ and taking its right derived functors: $$0 = R^q\Gamma(F) \cong H^q(\Gamma(F^{\bullet}))$$
Translating to our scenario, $F$ is an injective object in $Ab(\mathcal{E})$ with the above injective resolution. For each object $B$ in $\mathcal{E}$, we construct a left exact functor $B^*_{Ab} : Ab(\mathcal{E}) \rightarrow Ab(\mathcal{E} \downarrow B)$, as previously mentioned. Then $H^q(\mathcal{E}\downarrow B,B^*_{Ab}(F)) \cong R^q(B^*_{Ab}(F)) = 0$. Therefore, $F$ is flabby. 
\end{proof}

The following lemma is useful to prove the analogous  version of Proposition 3.2. We exhibit the proof provided in \cite{johnstone77topostheory} because it uses manipulations with geometric morphisms that show up every time we are working with Grothendieck Toposes.
\begin{lem}
Let $f: \mathcal{F} \rightarrow \mathcal{E}$ be a geometric morphism, with $\mathcal{E} = Sh(\mathcal{C},J)$, $F$ an object in $Ab(\mathcal{F})$, and $l: \mathcal{C} \rightarrow Sh(\mathcal{C},J)$ the canonical functor ($U \mapsto i^*(Hom(-,U))$). Then $R^qf_*(F)$ is the $J$-sheaf associated to the presheaf $U \mapsto H^q(\mathcal{F},f^*l(U);F)$.
\end{lem}
\begin{proof}
We split the proof of this lemma in two parts. First, we consider $J$ as minimal topology, and after that $J$ will be an arbitrary Grothendieck Topology.

The Grothendieck topology $J$ be minimal means $J(C) = \{ \text{maximal sieve in }C \}$, where $C$ is an object in $\mathcal{C}$. The minimal topology implies that $\mathcal{E} = Set^{\mathcal{C}^{op}}.$ Since $f$ is a geometric morphism, $f^*$ preservers finite limits and is left adjoint of $f_*$. So $f_*$ preserves small limits, $f_*(-)(U)$ is a left exact functor, and we can obtain the right derived functor $f_*(-)(U)$.
Besides that, by group cohomology definition and adjoint property of geometric morphism:
\begin{align*}
R^0f_*(-)(U) \cong f_*(-)(U) & \cong {}  Hom_{\mathcal{E}}(Hom(-,U),f_*(-))\\
      & \cong {} Hom_{\mathcal{F}}(f^*(Hom(-,U)),-) \\
      & \cong {} H^0(\mathcal{F},f^*(Hom(-,U)),-): Ab(\mathcal{F}) \rightarrow Ab({Sets})
\end{align*}
So the lemma holds for $J$ minimal.

Suppose $J$ is an arbitrary Grothendieck Topology in ${\cal C}$, let $i = (i_*, i^*) : \mathcal{E} \rightarrow Set^{\mathcal{C}^{op}}$ be the inclusion geometric morphism, and define $g = i \circ f: \mathcal{F} \rightarrow Set^{\mathcal{C}^{op}}$. The adjoint properties guarantees that $i^*g_* = (i^*i_*)f_* \cong f_*$. Since $i^*$ is an exact functor, $i^*R^qg_* \cong R^q(i^*g_*) \cong R^q(f_*)$. By the facts that $l: \mathcal{C} \rightarrow Sh(\mathcal{C},J)$ is the canonical functor and $i^*$ is the associated sheaf functor \cite[Chap III.5]{maclane1992sheaves}, we have  $$g^*(Hom(-,U)) = f^*i^*(Hom(-,U)) = f^*l(U)$$ Finally, we apply this in the calculations for $J$ minimal and conclude the desired result. 
\end{proof}

\begin{prop}
If $F$ is a flabby sheaf, then $R^qf_*(F) = 0,$ for all $q > 0$. In other words, $F$ is $f_*$-acyclic.
\end{prop}
\begin{proof}
We have $R^qf_*(F)$ is the $J$-sheaf associated to $U \mapsto H^q(\mathcal{F},f^*l(U);F)$, by the above Lemma. Since $F$ is flabby, $H^q(\mathcal{F},f^*l(U);F) = 0$ for all  $q > 0$. Thus, $R^qf_*(F) = 0,$ for all $q > 0$. 
\end{proof}

We also have a (Godement) resolution in this context \cite[page 265]{johnstone77topostheory}, and since the notion of flabby sheaf implies an acyclicity, we could use it to define cohomology groups using flabby sheaves instead of injective ones - see the discussion at the end of Section \ref{sec:6}. This approach is particularly interesting for cohomology in a topos because injectives resolutions depend on the axiom of choice to works properly while general toposes rely - internally - on intuitionistic logic. However, we observe that this definition of flabby does not coincide with the flabby definition for sheaves over topological spaces when $Sh(\mathcal{C},J) = Sh(X)$. Therefore, how we constructively generalize the flabby definition in $Sh(X)$ to $Sh(\mathcal{C},J)$? We do not know a definite answer to that but we will explain more about it in the last section.

\subsection{\v{C}ech Cohomology revisited}
\label{sec:13}

As expected, \v{C}ech Cohomoloy in the Grothendieck Topos case is more complicated. We will proceed more carefully now, and use some lemmas without proofs to not exceed in technicalities. 

We fix $\mathcal{E} = Sh(\mathcal{C},J)$, consider ${\cal P} = Set^{\mathcal{C}^{op}}$ it correspondent presheaves category, and  $i: \mathcal{E} \rightarrow {\cal P}$ the canonical inclusion. 

Suppose that $\mathcal{C}$ has pullbacks. For sheaves over topological spaces, when constructing the \v{C}ech Cohomology, we considered $U_{{i_0},...,{i_q}}$ as a   intersection of finite subfamilies of open sets that cover an open $U$. Now we need to find an analogous of this. Let $\mathcal{U}=(U_i \overset{f_i}\rightarrow U)_{i \in I} $ be a family of morphisms in $\mathcal{C}$, define\footnote{In other words, we select a specific pullback for each subfamily.}  $U_{{i_0},...,{i_q}} := U_{i_0} \times_U ... \times_U U_{i_q}$ (this is the \v{C}ech nerve). Applying morphisms $U_{{i_0},...,{i_q}} \xrightarrow{\delta_k} U_{{i_0},...,\widehat{i_k},...{i_q}}$ that ``forgets $i_k$'', we have a diagram in ${\cal P}$ as follows:

\begin{center}
\begin{tikzcd}
\dots \arrow[r, shift left=1 ex] \arrow[r, shift left=0.25 ex] 
\arrow[r, shift right=0.5 ex] 
& \coprod\limits_{i_0, i_1, i_2}h_{U_{i_0,i_1,i_2}} \arrow[r, shift left=1 ex] \arrow[r, shift right=0.5 ex]  & \coprod\limits_{i_0, i_1}h_{U_{i_0,i_1}}  \arrow[r] & h_U 
\end{tikzcd}  
\end{center}
where $h_U \cong Hom(-,U)$ is a representable functor. Remind that the forgetful functor $E: Ab({\cal P}) \rightarrow {\cal P}$ has a left adjoint $Z: {\cal P} \rightarrow Ab({\cal P})$, called free functor. Since left adjoint functors preserve colimits, we have a canonical isomorphism $Z(\coprod\limits_{j \in J}h_{V_j}) \cong \coprod\limits_{j \in J}Z(h_{V_j}) $.  

Apply the free functor in the diagram above to obtain a diagram in $Ab({\cal P})$:
\begin{center}
\begin{tikzcd}
\arrow[r, shift left=1 ex] \arrow[r, shift left=0.25 ex] 
\arrow[r, shift right=0.5 ex] 
& \coprod\limits_{i_0, i_1, i_2}Z(h_{U_{i_0,i_1,i_2}}) \arrow[r, shift left=1 ex] \arrow[r, shift right=0.5 ex]  & \coprod\limits_{i_0, i_1}Z(h_{U_{i_0,i_1}}) \arrow[r] & Z(h_U)
\end{tikzcd}  
\end{center}

Defining a boundary morphism $(d_q\alpha) = \sum\limits_{k=0}^{q+1}(-1)^k\alpha(\delta_k)_{\mid_{U_{{i_0},...,{i_{q+1}}}}}$, and using the above diagram, we construct a chain complex, denoted by $N_{\bullet}({\cal U})$, where $$N_q({\cal U}) := \coprod\limits_{i_0, i_1, ..., i_q}Z(h_{U_{i_0,i_1,...,i_q}})$$ Since, by \cite[Lemma 8.22]{johnstone77topostheory}, the sequence $\dots \rightarrow N_2({\cal U}) \rightarrow N_1({\cal U}) \rightarrow N_0({\cal U})$ is exact in ${Ab}({\cal P})$, we can use this chain complex to define the \textit{\v{C}ech cochain complex}.

Given a presheaf $F$ in $Ab({\cal P})$, the \textit{\v{C}ech cochain complex} is $$C^q(\mathcal{{U}},F) = Hom_{{Ab}({\cal P})}(N_q(\mathcal{U}),F),$$ with coboundary morphisms $d^q = - \circ d_q$. Since $(-\circ d_{q+1})\circ (-\circ d_{q}) = -\circ(d_q \circ d_{q+1}) = - \circ 0 = 0$, we define the $q$-th \v{C}ech cohomology group of $\mathcal{U}$ with coefficients in $F$ by $H^q(\mathcal{U},A) = Ker( d^{q})/Im(d^{q-1})$. 

Now, we want to define \v{C}ech Cohomology of an object in the category instead of its coverings.

Considering $\mathcal{V} = (V_j  \overset{g_j}\rightarrow U \enspace | \enspace j \in J)$ another  family of morphisms in $\mathcal{E}$
 that refines the family $\mathcal{U} = (U_i \overset{f_i}\rightarrow U \enspace | \enspace i \in I)$, we select a refinement map $r: \mathcal{V} \rightarrow \mathcal{U}$, that is,  a pair formed by a function $r : J \rightarrow I$ and  a family of factorizations $$\Bigg\{\begin{tikzcd}[cramped, sep=small] 
V_j \arrow[r, "r_j"] \arrow["g_j"', rd] & U_{r(j)} \arrow[d, "f_{r(j)}"] \\
 & U
\end{tikzcd} : j \in J\Bigg\}$$ 

In the following, we will abuse the notation and use $r_j$ to denote the value in the index set ($r_j = r(j) \in I$) and also the arrow in ${\cal C}$ ($r_j : V_j \to U_{r(j)}$). This will not cause confusion.

If $R$ is the sieve of $U$ generated by the family $\mathcal{U}$ (i.e., for any morphism $\alpha$ in $R$, $\alpha = f_i \circ h_i$, for some $i \in I$ and some $h_i$), then the inclusion map $\mathcal{U} \rightarrow R$  determines a refinement map.

\begin{prop}
Given $r, s: \mathcal{V} \rightarrow \mathcal{U}$ refinement maps, $r_{\bullet}$ and $s_{\bullet}$ are chain homotopic.
\end{prop}
\begin{proof}
We have to find a sequence of morphisms $N_q(\mathcal{V}) \rightarrow N_{q+1}(\mathcal{U})$ that makes $r_{\bullet}$ and $s_{\bullet}$ chain homotopics. 

Consider $ \sigma = (j_0,...,j_q)$ where $j_0,...,j_q \in J$. For each $l \in \{0,1,...,q\}$ we define a morphism over $U$ as follows:

$$t^l_{\sigma} = (r_{j_0},...,r_{j_l},s_{j_l},...,s_{j_q}): V_{\sigma} \rightarrow U_{(r_{j_0},...,r_{j_l},s_{j_l},...,s_{j_q})} $$

This morphism induces a ``group homomorphism'':
$$Z({t^l_\sigma}) : Z(h_{V_{\sigma}}) \rightarrow Z(h_{U_{(r_{j_0},...,r_{j_l},s_{j_l},...,s_{j_q})}}).$$

We can perform  two kind of actions:

(i) Consider the alternating sum of homomorphisms $t_\sigma := \sum_{l =0}^q (-1)^{l+1} \iota_{\sigma}^l \circ Z({t^l_\sigma})$, for this fixed $\sigma$, where
$$\iota_\sigma^l : Z(h_{U_{(r_{j_0},...,r_{j_l},s_{j_l},...,s_{j_q})}})  \to \coprod\limits_{i_0, i_1, ..., i_q, i_{q+1}}Z(h_{U_{i_0,i_1,...,i_q,i_{q+1}}})$$
is the canonical homomorphism.

(ii) ``Put together'' the homomorphisms $Z(t^l_\sigma)$, for a fixed $l$:
$$t^l_q : \coprod\limits_{j_0, j_1, ..., j_q}Z(h_{V_{j_0,j_1,...,j_q}}) \to \coprod\limits_{i_0, i_1, ..., i_q, i_{q+1}}Z(h_{U_{i_0,i_1,...,i_q,i_{q+1}}})$$
that we denote simply by
$t^l_q: N_q(\mathcal{V}) \rightarrow N_{q+1}(\mathcal{U})$

These two actions can be applied in any order we choose, without changing the resulting homomorphism, that we will denote by

$$t^{(q)}: N_q(\mathcal{V}) \rightarrow N_{q+1}(\mathcal{U})$$

Since $r$ and $s$ are refinement maps we can extract indices  $i_0,...,i_q \in I$ from $j_0, ..., j_q \in J$, thus we obtain an homomorphism

$$(r_q - s_q) : \coprod\limits_{j_0, j_1, ..., j_q}Z(h_{V_{j_0,j_1,...,j_q}}) \to \coprod\limits_{i_0, i_1, ..., i_q}Z(h_{U_{i_0,i_1,...,i_q}}) .$$

The (non commutative) diagram we must have in mind is:

\begin{center}
\begin{tikzcd}
\dots \arrow[r, "d^{\mathcal{V}}_3"] & {\coprod\limits_{j_0,j_1,j_2}Z(h_{V_{j_0,j_1,j_2}}}) \arrow[r, "d^{\mathcal{V}}_2"] \arrow[d, "r_2 - s_2"'] & {\coprod\limits_{j_0,j_1}Z(h_{V_{j_0,j_1}})} \arrow[d, "r_1 - s_1"] \arrow[ld, "t^{(1)}"'] \arrow[r, "d^{\mathcal{V}}_1"] & \dots \\
\dots \arrow[r, "d^{\mathcal{U}}_3"] & {\coprod\limits_{i_0,i_1,i_2}Z(h_{U_{i_0,i_1,i_2}}}) \arrow[r, "d^{\mathcal{U}}_2"] & {\coprod\limits_{i_0,i_1}Z(h_{U_{i_0,i_1}})} \arrow[r, "d^{\mathcal{U}}_1"] & \dots
\end{tikzcd}
\end{center}

We will exhibit the homotopy chain construction for the case $q = 2$: 

Let $\sigma = (j_0,j_1), \enspace  t^0_{\sigma} = (r_{j_0},s_{j_0},s_{j_1}), \enspace t^1_{\sigma} = (r_{j_0},r_{j_1},s_{j_1})$, and denote 
$$t_\sigma = \sum\limits_{l=0}^1(-1)^{l+1}\iota^l_\sigma \circ Z(t^l_\sigma) = - \iota^0_\sigma \circ Z(t^0_\sigma) + \iota^1_\sigma \circ Z(t^1_\sigma) := - (r_{j_0},s_{j_0},s_{j_1}) + (r_{j_0},r_{j_1},s_{j_1}).$$

Let $\tau = (j_0,j_1,j_2)$ and $\alpha_{\tau}: Z(h_{V_{\tau}}) \rightarrow \coprod\limits_{j_0, j_1, j_2}Z(h_{V_{j_0,j_1,j_2}})$. For an object $C$ in $\mathcal{C}$, we take $\theta^C_\tau  \in  Z(h_{V_{\tau}}(C))$ and obtain $$(r_2 - s_2) \circ \alpha_{\tau}(\theta_{\tau}) = \theta_{r_{i_0},r_{i_1},r_{i_2}} - \theta_{s_{i_0},s_{i_1},s_{i_2}}.$$ Apart from that, $d^{\mathcal{V}}_2(\alpha_{\tau}(\theta_{\tau})) = \theta_{j_0,j_1} - \theta_{j_0,j_2} + \theta_{j_1,j_2}$. Applying $t^{(1)}$ in the last equation:
\begin{align*}
    t^{(1)} \circ d^{\mathcal{V}}_2(\alpha_{\tau}(\theta_{\tau})) &= t^{(1)}(\theta_{j_0,j_1} - \theta_{j_0,j_2} + \theta_{j_1,j_2}) =\\
    &+(- \theta_{r_{j_0},s_{j_0},s_{j_1}} + \theta_{r_{j_0},r_{j_1},s_{j_1}}) \\
    &- (- \theta_{r_{j_0},s_{j_0},s_{j_2}} + \theta_{r_{j_0},r_{j_2},s_{j_2}}) \\
    &+ (- \theta_{r_{j_1},s_{j_1},s_{j_2}} + \theta_{r_{j_1},r_{j_2},s_{j_2}} )
\end{align*}

On the other hand:

\begin{align*}
    d^{\mathcal{U}}_3\circ t^{(2)}(\alpha_{\tau}(\theta_{\tau})) &= d^{\mathcal{U}}_3(-\theta_{r_{j_0},s_{j_0},s_{j_1},s_{j_2}}+\theta_{r_{j_0},r_{j_1},s_{j_1},s_{j_2}} -\theta_{r_{j_0},r_{j_1},r_{j_2},s_{j_2}}) \\
    &=-(\theta_{s_{j_0},s_{j_1},s_{j_2}} - \theta_{r_{j_0},s_{j_1},s_{j_2}}+\theta_{r_{j_0},s_{j_0},s_{j_2}}-\theta_{r_{j_0},s_{j_0},s_{j_1}})\\
    &+(\theta_{r_{j_1},s_{j_1},s_{j_2}}-\theta_{r_{j_0},s_{j_1},s_{j_2}}+\theta_{r_{j_0},r_{j_1},s_{j_2}}-\theta_{r_{j_0},r_{j_1},s_{j_1}})\\
    &-(\theta_{r_{j_1},r_{j_2},s_{j_2}}-\theta_{r_{j_0},r_{j_2},s_{j_2}}+\theta_{r_{j_0},r_{j_1},s_{j_2}}-\theta_{r_{j_0},r_{j_1},r_{j_2}})
\end{align*}

Therefore, $$d^{\mathcal{U}}_3\circ t^{(2)}(\alpha_{\tau}(\theta_{\tau})) + t^{(1)} \circ d^{\mathcal{V}}_2(\alpha_{\tau}(\theta_{\tau})) = \theta_{r_{i_0},r_{i_1},r_{i_2}}  - \theta_{s_{i_0},s_{i_1},s_{i_2}}
    =(r_2 - s_2) \circ \alpha_{\tau}(\theta_{\tau}).$$
    
So the chain homotopy is proved for $q=2$. For the general case, we consider $\tau = (j_0,...,j_q)$ then, by similar calculations, we obtain:

$$(r_q-s_q)(\alpha_{\tau}(\theta_{\tau})) = d^{\mathcal{U}}_{q+1}\circ t^{(q)}(\alpha_{\tau}(\theta_{\tau})) + t^{(q-1)} \circ d^{\mathcal{V}}_q(\alpha_{\tau}(\theta_{\tau}))$$
\end{proof}

The following result provides an isomorphism between cohomology groups of a family of morphism $\mathcal{U}$ and cohomology groups of a sieve R generated by $\mathcal{U}$.

\begin{prop}
Let $\mathcal{U} = (U_i \rightarrow U \enspace | \enspace i \in I)$ be a family of morphisms and $R$ the sieve generated by $\mathcal{U}$. Then the inclusion $i: \mathcal{U} \rightarrow R$ induces an isomorphism $\mathrm{H}^q(\mathcal{U},F) \cong  \mathrm{H}^q(R,F)$, for any presheaf $F$ in $Ab(\mathcal{P})$.
\end{prop}
\begin{proof}
Since $R$ is generated by $\mathcal{U},$ there is a refinement map $h : R \rightarrow {\cal U}$. On the other hand, we also have that the inclusion $i: \mathcal{U} \rightarrow R$ is a refinement map. By the previously proposition, this refinement is unique up to homotopy, thus $h \circ i$ and $i \circ h$ are cochain homotopic to the corresponding identity refinements. So, by Proposition \ref{homotopic},   $i$ induces a map in the cohomology group that is invertible. In other words,
 $\mathrm{H}^q(\mathcal{U},F) \cong  \mathrm{H}^q(R,F)$, canonically.

\end{proof}

If $\mathcal{C}$ has pullbacks we can define \v{C}ech cohomology groups of an object $U$ of $\mathcal{C}$ with coefficient in an abelian presheaf $F$ in $\mathcal{C}$ as the filtered colimit below
$$ \check{\mathrm{H}}^q(U,F) := \varinjlim\limits_{R \in J(U)} H^q(R,F) $$

Previously, we defined \v{C}ech Cohomology for a family of morphisms with the same codomain instead of considering sieves, but both cases are related: we can switch cover sieves with the family that generates it, by the above Proposition. 
We introduce this definition to obtain an analogous version of Theorem \ref{Cech-te} for Grothendieck topos.

\begin{teo}
Let $U$ be an object in $\mathcal{C}$ and $F$ a sheaf in $Ab(\mathcal{E})$. There is a homomorphism $k^q: \check{\mathrm{H}}^q(U,F) \rightarrow \mathrm{H}^q(\mathcal{E},l(U);F)$, $q \in \mathds{N}$, where $l: \mathcal{C} \rightarrow Sh(\mathcal{C},J)$ is the  canonical functor. Moreover,  $k^q$ is a isomorphism if $q = 0$ or $1$, and it is a monomorphism if $q=2$.
\end{teo}

To have an isomorphism in other cases we need to impose conditions on subsets of the set of objects in $\mathcal{C}$ as follows:

\begin{prop}
Let $\mathcal{E} = Sh(\mathcal{C},J)$, $F$ sheaf in $Ab(\mathcal{E})$. If there is a subset $K$ of the set of objects in $\mathcal{C}$ such that:
\begin{enumerate}
    \item[(i)] $\check{\mathrm{H}}^q(V,F) = 0, \forall q > 0$, for each $V \in K;$
    \item[(ii)] For each object $U$ in $\mathcal{C}$, there is a $J$-cover $\{V_j \overset{g_j}\rightarrow U \enspace | \enspace j \in J\}$ with $V_j \in K, \forall j \in J$; \item[(iii)] Every pullback of the form $V \times_U W$ is in $K$, whenever $V$ and $W$ are in $K$.
\end{enumerate}
Then the homomorphism $ k^q : \check{\mathrm{H}}^q(U;F) \rightarrow \mathrm{H}^q(\mathcal{E},l(U);F) $ is an isomorphism for any object $U$ in $\mathcal{C}$ and for all $q\in\mathds{N}$
\end{prop}

The proofs for both results above use spectral sequences and can be found at \cite[Chap 8]{johnstone77topostheory}. 

In particular, this last result can be applied to show the coincidence between sheaf and \v{C}ech cohomology in two cases, each one mentioned in \ref{sec:7} and \ref{sec:8}. Respectively, they are: 
\begin{enumerate}
    \item[(i)] $({\cal C},J)$ as the site canonically associated to a paracompact Hausdorff space $X$ and the coefficient $F$ as any sheaf of abelian groups in $Sh({\cal C}, J)$;
    \item[(ii)] $({\cal C},J)$ as the site canonically associated to a  scheme $(X, {\cal O}_X)$ and the coefficient sheaf $F$ as any quasi-coherent ${\cal O}_X$-module.
\end{enumerate}

\subsection{Applications}
\label{sec:14}

We already mentioned that Grothendieck topos cohomology was constructed to prove Weil’s conjectures. However, for this propose, Étale Cohomology is enough: there is no need to work with an arbitrary site $(\mathcal{C},J).$ If $\mathcal{C}$ is the slice category of schemes over a scheme $X$, where the objects are étale morphisms $Spec(R) \xrightarrow{f} X$, and, by abuse of notation, the morphism $f \xrightarrow{\varphi} g$ are morphisms of schemes $Spec(R) \xrightarrow{\varphi} Spec(R\text{\textquoteright}) $ such that $g \circ \varphi = f$.

Étale cohomology has good properties, e.g, can be related to singular cohomology, and has a Künneth formula, and Poincaré Duality with an adequate formulation. Furthermore, it has applications in number theory, $K$-theory, and representation theory of finite groups, besides its original use in algebraic geometry for fields different of $\mathds{C}$  and $\mathds{R}$. 

For other sites, we obtain other cohomologies such as crystalline, Deligne, and flat cohomologies. They also are instances of the Grothendieck topos cohomology we presented.

There are other kinds of applications of Grothendieck topos cohomology. If $\mathcal{C}$ is a small category, and $F$ is an abelian presheaf in $Ab(Set^{C^{op}})$, we can define a cochain complex $C^q(\mathcal{C},F) = \prod\limits_{c_0\leftarrow...\leftarrow c_q } F(c_q)$ with an appropriate coboundary $d^{q}: C^{q}(\mathcal{C},F) \rightarrow C^{q+1}(\mathcal{C},F)$, to obtain $H^{q}(\mathcal{C},F) = Ker(d^q)/Im(d^{q-1})$ as the cohomology groups of the category $\mathcal{C}$ with coefficients in $F$. Then, we have an isomorphism  $H^{\bullet}(\mathcal{C},F) \cong H^{\bullet}(Set^{\mathcal{C}^{op}},F)$. For a proof of this and an explicit description of the coboundary maps, consult \cite[Chap.II.6]{moerdijk1995classifying}. 

A simple example of this isomorphism manifests when the presheaf is $Set^G$, where $G$ is group seen as a category with a single object. In such case, the objects in $Ab(Set^G)$ are right modules over the group ring $\mathds{Z}G$. Thus, the cohomology groups of $G$ obtained from group cohomology are isomorphic to the sheaf cohomology groups of  $Set^{\mathcal{C}^{op}}$. This is better know in the form $H^{\bullet}(BG,M) \cong H^{\bullet}(G,M)$, where $BG$ is the classifying space of $G$ and $M$ is a $G$-module. Consult \cite[Chap. II]{adem2013cohomology} to see the usual approach. 
Furthermore, given a topological group $G$, there is a natural and useful variation of the formerly mentioned group cohomology but defined on the abelian category ${\cal A}_G$ of all {\em discrete $G$-modules for a continuous $G$-action} - in particular, this is the case studied in profinite cohomology, that encompasses Galois Cohomology. If ${\cal E}_G$ is the category of all sets (i.e., discrete spaces) endowed with a continuous $G$-action and $G$-equivariant maps, then: ${\cal E}_G$ is a Grothendieck topos, the expected equivalence $Ab({\cal E}_G) \simeq {\cal A}_G$ holds, and the cohomology of the topos ${\cal E}_G$ coincides with the above described {\em continuous} cohomology of $G$.

Hence, Grothendieck topos cohomology also is related to non-sheaf cohomology, and not only with cohomology for specific sites. We will provide further applications in the next section.

\section{Remarks and New Frontiers}
\label{sec:15}

Topos are excellent environments for internalizing mathematical objects, and we can write formulas for a language (type theory)
like arrows hitting the subobject classifier.
For example, to each formula  $\phi(x)$  with a free variable $x$ of type $ X $ is
associated with the  subobject of $X$ that classically corresponds to
“$ \{x \in X \mid \phi (x) \} $”. In this way, we can interpret a
high order type theory in a topos via the so-called semantics of
Kripke-Joyal. Results on elementary topos include that they are finitely co-complete, represent the idea of ``parts of an object'' and that its internal logic is intuitionist and, in particular, the parts of an object define an internal Heyting algebra. So, a topos is an environment
for higher order intuitionist mathematics --- evidently not all the topos
 are equivalent, so there is a diversity of environments.

Daily mathematics makes use of set theories to represent higher-order aspects of mathematical theories: this can be understood as the use of the higher-order internal logic of the $ Set $ topos. Since the 1970s, mathematical applications of higher-order intuitionist internal logic approaches have been applied to topos:\\
(i) an internal approach to the Serre-Swan duality, through a simple theorem, was described in \cite{mulvey1974intuitionistic} (essentially) of Linear Algebra, Kaplansky's Theorem \footnote{Every module on a local ring that is projective and finitely generated is a free module.}; \\
(ii) in model constructions, via Grothendieck topos, of synthetic differential geometry (\cite{moerdijk1991models}, \cite{mclarty1992elementary}): for instance, in \cite{moerdijk1991models}, there is an internal version of de Rham Theorem (a deep connection between de Rham cohomology and singular homology);\\
(iii) to represent results of quantum mechanics as results of classical mechanics internal to a topos \cite{flori2013first}; \\
(iv) in algebraic geometry, although the origin of Grothendieck's notion of topos came from specific needs of algebraic geometry, more systematic explorations of the internal language of topos for this area are very recent: e.g.,  \cite{blechschmidt2018using} contains a dictionary between the external and the internal point of view (for example, objects in a topos, are, internally, just sets; monomorphism are injections; sheaves of rings are rings),  works with the big and small Zariski Topos associated to a scheme to exhibit simpler definitions and proofs by using the internal language provided by these toposes, and in \cite{blechschmidt2018flabby} explores a proposal of a constructive version of the main homological tools (flabby and injective objects).

Attempts to develop constructive approaches to homological algebra, without the aid of axiom of choice (as in the usual injective resolution construction), are different from ``cohomology in topos'', although they can be related. In the latter case, we usually are interested in a Grothendieck topos, and constructing cohomology groups with coefficient in $Ab(Sh(\mathcal{C},J)),$ for some site $(\mathcal{C},J)$. That is exactly what we exposed in the previous section, using P. Johnstone's book ``Topos Theory'' \cite{johnstone77topostheory}, as the main reference. However, similar to the extension of sheaf cohomology to Grothendieck topos cohomology, how could we extend Grothendieck topos cohomology to (elementary) topos cohomology? The first problem is that for an elementary topos $\mathcal{E}$ we can not guarantee that $Ab(\mathcal{\mathcal{E}})$ have enough injectives so it is not clear how to define the right derived functors - there is a form to construct them using noetherian abelian categories \cite{WILDESHAUS2000207}, but we do not know any systematically study to identify when $Ab(\mathcal{E})$ is noetherian abelian. Still in the topic of ``topos cohomology'' we could try to switch $Ab(\mathcal{\mathcal{E}})$ for $\mathcal{E}$. In this direction, we have the work of I. Blechschmidt that is closely related to develop a constructive version of homological algebra: he reintroduces the concepts of injective object and flabby sheaf, as objects in an elementary topos, and replaces injective resolutions with flabby resolutions to avoid the use of the axiom of choice. However, in the final chapter of his article, he calls attention to the open problem of how to embed an arbitrary sheaf of modules into a flabby sheaf in intuitionistic logic. We understand that this way of proceeding (defining objects inside a topos) was successfully adopted before in another context, by A. Grothendieck, when he defined the fundamental group on a topos and originated the ``Grothendieck's Galois Theory'' \cite{grothendieck1971revetement}. The theory was later extended by A. Joyal and M. Tierney in \cite{joyal1984extension}. The results of this latter article are constantly used in nowadays works, which indicates that studies in the same direction for homological algebra would provide important discoveries.

Pertinent to this discussion, we can cite Blass's work that shows cohomology can detect the failure of the axiom of choice \cite{blass1983cohomology}. He demonstrates the axiom of choice is equivalent to $H^1(X,G) = 0$, for all discrete set $X$, and all group $G$. Also, the triviality of $H^1(X,G)$ for all $G$ is equivalent to the projectivity of $X$. This strengthens the relation between logic and geometry that we have been pointing through toposes. Note Blass's results indicate a justification for the fact that $Ab(\mathcal{E})$ does not have enough projectives, in general, because of toposes' intuitionistic logic.

We believe the subject of ``topos cohomology'' is far from maturity. One of the main references into the subject, SGA4, only addresses the case $Ab(Sh(\mathcal{C},J))$. P. Johnstone, one of the most prominent topos theorists of our days, had not published the third volume of ``Sketches of an Elephant: A Topos Theory Compendium'' that would contain the subject of homotopy and cohomology in toposes (besides chapters about toposes as mathematical universes) and the first two volumes were released in 2002 \cite{johnstone2002sketches,johnstone2002sketches2}.

Regarding constructive methods for homological algebra, there also are investigations not involving toposes. For example, in \cite{ubsi_1179}, S. Posur's provides constructive methods in the context of abelian categories using generalized morphisms (we highlight it is not the same definition given by S. MacLane in \cite{zbMATH01216133}). He proves the Snake Lemma, establish what are generalized cochain complex and generalized homological groups, and present a notion of homological group in a concrete way, i.e., he displays explicitly the connecting morphism, and not only states it exists by universal properties. More than that, he applies the theoretical definitions to create an algorithm capable of computing spectral sequences for a certain abelian category, and use it to calculate cohomology groups of (specific) equivariant sheaves.

 In recent work (\cite{shulman2018linear}), M. Schulman defends that ``Linear Logic'' can clarify some constructive methods better than intuitionistic logic. We highlight Schulman's state that it provide constructivist definitions (and proofs) of concepts elaborated in classical logic. Then a ``linear approach'' could also be useful for the problems we mentioned concerning constructive cohomology. Furthermore, generalized metric spaces (or quasi-psudo-metric space, or Lawvere metric space \cite{lawvere1973metric}) can be redefined using linear logic \cite{shulman2018linear}.

Linear Logic is a weakening of intuitionistic logic: it is a ``sub-structural'' logic, i.e., the usual demonstrability rules do not apply in general, with only restricted versions of the contraction and weakening rules available.
In linear logic, the intuitionist conjunction  splits into two
binary operators: $\wedge$,
the  binary infimum of the lattice, which is not necessarily
distributes over the supreme; $\&$,
another operation that does distribute with arbitrary supreme,
but it doesn't have to be idempotent or commutative.

The study of linear logic was initially developed by Jean-Yves Girard \cite{girard1987linear}
in the context of polymorphic $ \lambda $ calculus, but its nature matches
- through splits - somewhat irreconcilable elements, and their
many interpretations have profound meaning. Pure
intuitionistic contexts cannot prove the excluded middle law and, in 
classical logic, this is nothing less than an axiom. Linear logic has two
candidates for disjunction $ \lor $, one for which it is impossible to prove the 
excluded middle, and another for which the evidence is trivial.

The presence of “duplication” of operators is natural, as these
represent useful fragments of the usual logical operations. The result
interesting is related to the famous correspondence of
Curry-Howard: in the same way that intuitionistic logic is related
with type theory and $ \lambda $ calculation simply typed (the
implication can be interpreted as the type of functions, conjunction with
product and disjunction with co-product) giving rise to “proof-relevance”,
linear logic introduces, via non-idempotency or non-commutativity,
the relationship of linear implication to processes that are
“Resource-relevant”.

Categorical semantics for various forms of linear logics have long been explored (e.g.   \cite{seely1987linear}, \cite{hyland1993full}). Roughly speaking, we can say that closed monoidal categories have (some form of) internal linear logic.

Something very different occurs when we focus on possible conjunctistic or higher-order aspects (\cite{lambekscott1986catlog}, \cite{bell1988topos}) that are internal to a special type of category governed by some form of linear logic.

A natural, and relatively simple, way to expand the notion of (categories of) sheaves  with internal logic that is no longer intuitionistic is through appropriate adaptations of the sheaf notion defined over a complete Heyting algebra $(H, \leq, \wedge) $ to other algebras that are also complete lattices.

These set-theoretical aspects of the sheaves on ``good'' complete lattices can also be approached in an alternative, but often ``equivalent'' way, through the notion of expansion of the  universe of all sets, $V$, by an algebra, $A$, which is a complete lattice, $ V^{(A)} $: in the (traditional) case where $ A $ is a Boolean algebra or Heyting algebra this is presented in \cite{bell2005boolean}.

 The  complete lattices that have natural relationship with linear logics are the quantales  (see \cite{yetter1990quantal}). A quantale  $ (Q, \leq, \otimes, \top) $ is a structure where: $ (Q, \leq) $ is a complete lattice where $\top$ is the top element, $ (Q, \otimes) $ is a semigroup and the distributive laws are valid: $ a  \otimes \bigvee_{i \in I} b_i = \bigvee_{i \in I} a \otimes b_i $, $ (\bigvee_{i \in I} b_i) \otimes a = \bigvee_{i \in I} b_i \otimes a $.
 
 There are some early explorations of
the strategy of considering ``generalized sheaves'', with  applications in Mathematics. In \cite{con97}, is established a notion of  category of ``sheaves'' over a quantale  $ (Q, \leq, \otimes, \top) $, which is right-sided ($ a \otimes \top = a, a \in Q $) and idempotent ($ a \otimes a = a, a \in Q $), and is explored  the above mentioned Kaplansky's Theorem, now reformulated in the internal linear logic of this category of sheaves. In the work \cite{MeMa},  categories of sheaves are considered over quantales $ (Q, \leq, \otimes, \top) $ satisfying a different balance: they are commutative and semicartesian (or two-sided\footnote{Since it is already commutative,  this is the same that require to be right-sided.}). It is important to emphasize that the two-sided, commutative, and {\em idempotent} quantales coincides with the complete Heyting algebras.

In \cite{MeMa}, given a commutative semicartesian quantale $(Q,\leq,\odot, 1)$, we can construct what is called a $Q$-$set$, in the same spirit of the construction of sheaves over complete Heyting algebras. These $Q$-$set$ will not provide a sheaf, but will preserve a significant part of a sheaf structure, which had motivated the authors to called it a ``Sheaf-Like category'', besides that, pseudo-metric spaces are examples of a $Q$-$set$ (when $Q = ([0,\infty],\geq, +, 0)$), and also of an enriched category over $Q$. This approach seems to expand the development of model theory of Continuous Logic, useful in Functional Analysis.

In \cite{leinster2017magnitude}, M. Schulman and T. Leinster use semicartesian monoidal categories $V$ to define magnitude homology of $V$-categories (enriched categories over $V$). In particular, if $V$ is the extended non-negative real numbers $[0, \infty]$, that admits a natural structure of a commutative semicartesian  quantale  $([0,\infty],\geq, +, 0)$, then the correspondent $V$-category is a generalized metric space. Magnitude homology describes a general notion of ``size''. Depending on the case, it coincides with the cardinality of a set, the Euler Characteristic of topological space, or of an associate algebra. For the metric space context, magnitude machinery provide interesting geometric properties as area \cite{willerton2014magnitude}, volume \cite{barcelo2018magnitudes}, and Minkowski dimension \cite{meckes2015magnitude}.

This conjuncture motivates the authors of this survey to wonder about internal cohomological aspects to other categories governed by other logics. In particular: (i) if developments in continuous model theory -for instance with applications to the theory of Banach algebras- could have internal cohomological aspects better represented in the linear logic style; (ii) if exploring metric spaces as enriched categories over a quantale, it is natural to consider possible connections between  magnitude (co)homology with an adapted sheaf-like cohomology by some  appropriate version of sheaves over quantales.\\

\textbf{Acknowledgements:} The comments of  Peter Arndt and Walter de Siqueira Pedra, members of the judging committee in the master dissertation (supported by CNPq) of the first author, under supervision by the second author, guided for paths of study we would not perceive alone, so improving this text with more applications and related lines of research. We are thankful for it.

%

\bibliographystyle{abbrv}
\bibliography{bibliography}  

\end{document}